\documentclass[11pt]{amsart}

\usepackage{amsmath}
\usepackage{amssymb}
\usepackage{amscd}
\usepackage{color}
\usepackage{hyperref}
\usepackage{mathrsfs}

\usepackage{eucal}

\usepackage[normalem]{ulem}

\usepackage{extpfeil, extarrows } 

\usepackage{xy}
\xyoption{all}

\newcommand{\nc}{\newcommand}

\renewcommand{\AA}{{\mathbb{A}}}
\nc{\CC}{{\mathbb{C}}}
\nc{\LL}{{\mathbb{L}}}
\nc{\RR}{{\mathbb{R}}}
\renewcommand{\P}{{\mathbb{P}}}
\nc{\OO}{{\mathbb{O}}}

\nc{\QQ}{{\mathbb{Q}}}
\nc{\ZZ}{{\mathbb{Z}}}

\nc{\cA}{{\mathcal{A}}}
\nc{\cB}{{\mathcal{B}}}
\nc{\cC}{{\mathcal{C}}}
\nc{\cD}{{\mathcal{D}}}
\nc{\tcD}{{\tilde{\mathcal{D}}}}
\nc{\scD}{{\mathscr{D}}}
\nc{\cE}{{\mathcal{E}}}
\nc{\cF}{{\mathcal{F}}}
\nc{\cG}{{\mathcal{G}}}
\nc{\cH}{{\mathcal{H}}}
\nc{\cI}{{\mathcal{I}}}
\nc{\cJ}{{\mathcal{J}}}
\nc{\cK}{{\mathcal{K}}}
\nc{\cL}{{\mathcal{L}}}
\nc{\cM}{{\mathcal{M}}}
\nc{\cN}{{\mathcal{N}}}
\nc{\cO}{{\mathcal{O}}}
\nc{\cP}{{\mathcal{P}}}
\nc{\cQ}{{\mathcal{Q}}}
\nc{\cR}{{\mathcal{R}}}
\nc{\cS}{{\mathcal{S}}}
\nc{\cT}{{\mathcal{T}}}
\nc{\cU}{{\mathcal{U}}}
\nc{\cV}{{\mathcal{V}}}
\nc{\cW}{{\mathcal{W}}}
\nc{\cX}{{\mathcal{X}}}
\nc{\cY}{{\mathcal{Y}}}
\nc{\cZ}{{\mathcal{Z}}}

\nc{\rc}{{\mathrm{c}}}

\nc{\rQ}{{\mathrm{Q}}}
\nc{\rS}{{\mathrm{S}}}

\nc{\bA}{{\mathbf{A}}}
\nc{\bB}{{\mathbf{B}}}
\nc{\bC}{{\mathbf{C}}}
\nc{\bD}{{\mathbf{D}}}
\nc{\bE}{{\mathbf{E}}}
\nc{\bF}{{\mathbf{F}}}
\nc{\bG}{{\mathbf{G}}}
\nc{\bH}{{\mathbf{H}}}
\nc{\bI}{{\mathbf{I}}}
\nc{\bJ}{{\mathbf{J}}}
\nc{\bK}{{\mathbf{K}}}
\nc{\bL}{{\mathbf{L}}}
\nc{\bM}{{\mathbf{M}}}
\nc{\bN}{{\mathbf{N}}}
\nc{\bO}{{\mathbf{O}}}
\nc{\bP}{{\mathbf{P}}}
\nc{\bQ}{{\mathbf{Q}}}
\nc{\bR}{{\mathbf{R}}}
\nc{\bS}{{\mathbf{S}}}
\nc{\bT}{{\mathbf{T}}}
\nc{\bU}{{\mathbf{U}}}
\nc{\bV}{{\mathbf{V}}}
\nc{\bW}{{\mathbf{W}}}
\nc{\bX}{{\mathbf{X}}}
\nc{\bY}{{\mathbf{Y}}}
\nc{\bZ}{{\mathbf{Z}}}

\nc{\ba}{{\mathbf{a}}}
\nc{\bb}{{\mathbf{b}}}
\nc{\bc}{{\mathbf{c}}}
\nc{\bd}{{\mathbf{d}}}
\nc{\be}{{\mathbf{e}}}
\nc{\bg}{{\mathbf{g}}}
\nc{\bh}{{\mathbf{h}}}
\nc{\bi}{{\mathbf{i}}}
\nc{\bj}{{\mathbf{j}}}
\nc{\bk}{{\mathbf{k}}}
\nc{\bl}{{\mathbf{l}}}
\nc{\bm}{{\mathbf{m}}}
\nc{\bn}{{\mathbf{n}}}
\nc{\bo}{{\mathbf{o}}}
\nc{\bp}{{\mathbf{p}}}
\nc{\bq}{{\mathbf{q}}}
\nc{\br}{{\mathbf{r}}}
\nc{\bs}{{\mathbf{s}}}
\nc{\bt}{{\mathbf{t}}}
\nc{\bu}{{\mathbf{u}}}
\nc{\bv}{{\mathbf{v}}}
\nc{\bw}{{\mathbf{w}}}
\nc{\bx}{{\mathbf{x}}}
\nc{\by}{{\mathbf{y}}}
\nc{\bz}{{\mathbf{z}}}

\nc{\fA}{{\mathfrak{A}}}
\nc{\fB}{{\mathfrak{B}}}
\nc{\fC}{{\mathfrak{C}}}
\nc{\fD}{{\mathfrak{D}}}
\nc{\fE}{{\mathfrak{E}}}
\nc{\fF}{{\mathfrak{F}}}
\nc{\fG}{{\mathfrak{G}}}
\nc{\fH}{{\mathfrak{H}}}
\nc{\fI}{{\mathfrak{I}}}
\nc{\fJ}{{\mathfrak{J}}}
\nc{\fK}{{\mathfrak{K}}}
\nc{\fL}{{\mathfrak{L}}}
\nc{\fM}{{\mathfrak{M}}}
\nc{\fN}{{\mathfrak{N}}}
\nc{\fO}{{\mathfrak{O}}}
\nc{\fP}{{\mathfrak{P}}}
\nc{\fQ}{{\mathfrak{Q}}}
\nc{\fR}{{\mathfrak{R}}}
\nc{\fS}{{\mathfrak{S}}}
\nc{\fT}{{\mathfrak{T}}}
\nc{\fU}{{\mathfrak{U}}}
\nc{\fV}{{\mathfrak{V}}}
\nc{\fW}{{\mathfrak{W}}}
\nc{\fX}{{\mathfrak{X}}}
\nc{\fY}{{\mathfrak{Y}}}
\nc{\fZ}{{\mathfrak{Z}}}

\nc{\fa}{{\mathfrak{a}}}
\nc{\fb}{{\mathfrak{b}}}
\nc{\fc}{{\mathfrak{c}}}
\nc{\fd}{{\mathfrak{d}}}
\nc{\fe}{{\mathfrak{e}}}
\nc{\ff}{{\mathfrak{f}}}
\nc{\fg}{{\mathfrak{g}}}
\nc{\fh}{{\mathfrak{h}}}
\nc{\fj}{{\mathfrak{j}}}
\nc{\fk}{{\mathfrak{k}}}
\nc{\fl}{{\mathfrak{l}}}
\nc{\fm}{{\mathfrak{m}}}
\nc{\fn}{{\mathfrak{n}}}
\nc{\fo}{{\mathfrak{o}}}
\nc{\fp}{{\mathfrak{p}}}
\nc{\fq}{{\mathfrak{q}}}
\nc{\fr}{{\mathfrak{r}}}
\nc{\fs}{{\mathfrak{s}}}
\nc{\ft}{{\mathfrak{t}}}
\nc{\fu}{{\mathfrak{u}}}
\nc{\fv}{{\mathfrak{v}}}
\nc{\fw}{{\mathfrak{w}}}
\nc{\fx}{{\mathfrak{x}}}
\nc{\fy}{{\mathfrak{y}}}
\nc{\fz}{{\mathfrak{z}}}

\nc{\sA}{{\mathsf{A}}}
\nc{\sB}{{\mathsf{B}}}
\nc{\sC}{{\mathsf{C}}}
\nc{\sD}{{\mathsf{D}}}
\nc{\sE}{{\mathsf{E}}}
\nc{\sF}{{\mathsf{F}}}
\nc{\sG}{{\mathsf{G}}}
\nc{\sH}{{\mathsf{H}}}
\nc{\sI}{{\mathsf{I}}}
\nc{\sJ}{{\mathsf{J}}}
\nc{\sK}{{\mathsf{K}}}
\nc{\sL}{{\mathsf{L}}}
\nc{\sM}{{\mathsf{M}}}
\nc{\sN}{{\mathsf{N}}}
\nc{\sO}{{\mathsf{O}}}
\nc{\sP}{{\mathsf{P}}}
\nc{\sQ}{{\mathsf{Q}}}
\nc{\sR}{{\mathsf{R}}}
\nc{\sS}{{\mathsf{S}}}
\nc{\sT}{{\mathsf{T}}}
\nc{\sU}{{\mathsf{U}}}
\nc{\sV}{{\mathsf{V}}}
\nc{\sW}{{\mathsf{W}}}
\nc{\sX}{{\mathsf{X}}}
\nc{\sY}{{\mathsf{Y}}}
\nc{\sZ}{{\mathsf{Z}}}

\nc{\sa}{{\mathsf{a}}}
\nc{\sd}{{\mathsf{d}}}
\nc{\se}{{\mathsf{e}}}
\nc{\sg}{{\mathsf{g}}}
\nc{\sh}{{\mathsf{h}}}
\nc{\si}{{\mathsf{i}}}
\nc{\sj}{{\mathsf{j}}}
\nc{\sk}{{\mathsf{k}}}
\nc{\sm}{{\mathsf{m}}}
\nc{\sn}{{\mathsf{n}}}
\nc{\so}{{\mathsf{o}}}
\nc{\sq}{{\mathsf{q}}}
\nc{\sr}{{\mathsf{r}}}
\nc{\st}{{\mathsf{t}}}
\nc{\su}{{\mathsf{u}}}
\nc{\sv}{{\mathsf{v}}}
\nc{\sw}{{\mathsf{w}}}
\nc{\sx}{{\mathsf{x}}}
\nc{\sy}{{\mathsf{y}}}
\nc{\sz}{{\mathsf{z}}}

\nc{\oA}{{\overline{A}}}
\nc{\oB}{{\overline{B}}}
\nc{\oC}{{\overline{C}}}
\nc{\oD}{{\overline{D}}}
\nc{\oE}{{\overline{E}}}
\nc{\oF}{{\overline{F}}}
\nc{\oG}{{\overline{G}}}
\nc{\oH}{{\overline{H}}}
\nc{\oI}{{\overline{I}}}
\nc{\oJ}{{\overline{J}}}
\nc{\oK}{{\overline{K}}}
\nc{\oL}{{\overline{L}}}
\nc{\oM}{{\overline{M}}}
\nc{\oN}{{\overline{N}}}
\nc{\oO}{{\overline{O}}}
\nc{\oP}{{\overline{P}}}
\nc{\oQ}{{\overline{Q}}}
\nc{\oR}{{\overline{R}}}
\nc{\oS}{{\overline{S}}}
\nc{\oT}{{\overline{T}}}
\nc{\oU}{{\overline{U}}}
\nc{\oV}{{\overline{V}}}
\nc{\oW}{{\overline{W}}}
\nc{\oX}{{\overline{X}}}
\nc{\oY}{{\overline{Y}}}
\nc{\oZ}{{\overline{Z}}}

\nc{\oa}{{\overline{a}}}
\nc{\ob}{{\overline{b}}}
\nc{\oc}{{\overline{c}}}
\nc{\od}{{\overline{d}}}
\nc{\of}{{\overline{f}}}
\nc{\og}{{\overline{g}}}
\nc{\oh}{{\overline{h}}}
\nc{\oi}{{\overline{i}}}
\nc{\oj}{{\overline{j}}}
\nc{\ok}{{\overline{k}}}
\nc{\ol}{{\overline{l}}}
\nc{\om}{{\overline{m}}}
\nc{\on}{{\overline{n}}}
\nc{\oo}{{\overline{o}}}
\nc{\op}{{\mathrm{op}}}
\nc{\oq}{{\overline{q}}}
\nc{\os}{{\overline{s}}}
\nc{\ot}{{\overline{t}}}
\nc{\ou}{{\overline{u}}}
\nc{\ov}{{\overline{v}}}
\nc{\ow}{{\overline{w}}}
\nc{\ox}{{\overline{x}}}
\nc{\oy}{{\overline{y}}}
\nc{\oz}{{\overline{z}}}

\nc{\tA}{{\tilde{A}}}
\nc{\tB}{{\tilde{B}}}
\nc{\tC}{{\tilde{C}}}
\nc{\tD}{{\tilde{D}}}
\nc{\tE}{{\tilde{E}}}
\nc{\tF}{{\tilde{F}}}
\nc{\tG}{{\tilde{G}}}
\nc{\tH}{{\tilde{H}}}
\nc{\tI}{{\tilde{I}}}
\nc{\tJ}{{\tilde{J}}}
\nc{\tK}{{\tilde{K}}}
\nc{\tL}{{\tilde{L}}}
\nc{\tM}{{\tilde{M}}}
\nc{\tN}{{\tilde{N}}}
\nc{\tO}{{\tilde{O}}}
\nc{\tP}{{\tilde{P}}}
\nc{\tQ}{{\tilde{Q}}}
\nc{\tR}{{\tilde{R}}}
\nc{\tS}{{\tilde{S}}}
\nc{\tT}{{\tilde{T}}}
\nc{\tU}{{\tilde{U}}}
\nc{\tV}{{\tilde{V}}}
\nc{\tW}{{\tilde{W}}}
\nc{\tX}{{\tilde{X}}}
\nc{\tY}{{\tilde{Y}}}
\nc{\tZ}{{\tilde{Z}}}

\nc{\ta}{{\tilde{a}}}
\nc{\tb}{{\tilde{b}}}
\nc{\tc}{{\tilde{c}}}
\nc{\td}{{\tilde{d}}}
\nc{\te}{{\tilde{e}}}
\nc{\tf}{{\tilde{f}}}
\nc{\tg}{{\tilde{g}}}
\nc{\ti}{{\tilde{i}}}
\nc{\tj}{{\tilde{j}}}
\nc{\tk}{{\tilde{k}}}
\nc{\tl}{{\tilde{l}}}
\nc{\tm}{{\tilde{m}}}
\nc{\tn}{{\tilde{n}}}
\nc{\tp}{{\tilde{p}}}
\nc{\tq}{{\tilde{q}}}
\nc{\tr}{{\tilde{r}}}
\nc{\ts}{{\tilde{s}}}
\nc{\tu}{{\tilde{u}}}
\nc{\tv}{{\tilde{v}}}
\nc{\tw}{{\tilde{w}}}
\nc{\tx}{{\tilde{x}}}
\nc{\ty}{{\tilde{y}}}
\nc{\tz}{{\tilde{z}}}

\nc{\hA}{{\hat{A}}}
\nc{\hB}{{\hat{B}}}
\nc{\hC}{{\hat{C}}}
\nc{\hD}{{\hat{D}}}
\nc{\hE}{{\hat{E}}}
\nc{\hF}{{\hat{F}}}
\nc{\hG}{{\hat{G}}}
\nc{\hH}{{\hat{H}}}
\nc{\hI}{{\hat{I}}}
\nc{\hJ}{{\hat{J}}}
\nc{\hK}{{\hat{K}}}
\nc{\hL}{{\hat{L}}}
\nc{\hM}{{\hat{M}}}
\nc{\hN}{{\hat{N}}}
\nc{\hO}{{\hat{O}}}
\nc{\hP}{{\hat{P}}}
\nc{\hQ}{{\hat{Q}}}
\nc{\hR}{{\hat{R}}}
\nc{\hS}{{\hat{S}}}
\nc{\hT}{{\hat{T}}}
\nc{\hU}{{\hat{U}}}
\nc{\hV}{{\hat{V}}}
\nc{\hW}{{\hat{W}}}
\nc{\hX}{{\hat{X}}}
\nc{\hY}{{\hat{Y}}}
\nc{\hZ}{{\hat{Z}}}

\nc{\ha}{{\hat{a}}}
\nc{\hb}{{\hat{b}}}
\nc{\hc}{{\hat{c}}}
\nc{\hd}{{\hat{d}}}
\nc{\he}{{\hat{e}}}
\nc{\hf}{{\hat{f}}}
\nc{\hg}{{\hat{g}}}
\nc{\hh}{{\hat{h}}}
\nc{\hi}{{\hat{i}}}
\nc{\hj}{{\hat{j}}}
\nc{\hk}{{\hat{k}}}
\nc{\hl}{{\hat{l}}}
\nc{\hm}{{\hat{m}}}
\nc{\hn}{{\hat{n}}}
\nc{\ho}{{\hat{o}}}
\nc{\hp}{{\hat{p}}}
\nc{\hq}{{\hat{q}}}
\nc{\hr}{{\hat{r}}}
\nc{\hs}{{\hat{s}}}
\nc{\hu}{{\hat{u}}}
\nc{\hv}{{\hat{v}}}
\nc{\hw}{{\hat{w}}}
\nc{\hx}{{\hat{x}}}
\nc{\hy}{{\hat{y}}}
\nc{\hz}{{\hat{z}}}

\nc{\eps}{\varepsilon}
\nc{\lan}{\big\langle}
\nc{\ran}{\big\rangle}
\nc{\kk}{{\Bbbk}}

\def\bw#1#2{\textstyle{\bigwedge\hskip-0.9mm^{#1}}\hskip0.2mm{#2}}

\DeclareMathOperator{\Hom}{\mathrm{Hom}}
\DeclareMathOperator{\Ext}{\mathrm{Ext}}

\DeclareMathOperator{\cRHom}{\mathrm{R}\mathcal{H}\!\mathit{om}}
\DeclareMathOperator{\Tor}{\mathrm{Tor}}

\DeclareMathOperator{\Spec}{\mathrm{Spec}}

\DeclareMathOperator{\Bl}{\mathrm{Bl}}
\DeclareMathOperator{\Sing}{\mathrm{Sing}}

\DeclareMathOperator{\Br}{\mathrm{Br}}

\DeclareMathOperator{\Ker}{\mathrm{Ker}}

\DeclareMathOperator{\Cone}{\mathrm{Cone}}

\DeclareMathOperator{\Gr}{\mathrm{Gr}}

\DeclareMathOperator{\id}{\mathrm{id}}

\DeclareMathOperator{\codim}{\mathrm{codim}}

\DeclareMathOperator{\pr}{\mathrm{pr}}
\DeclareMathOperator{\Fun}{\mathrm{Fun}}

\newcommand{\Db}{\mathrm{D}^{\mathrm{b}}}

\newcommand{\Dqc}{\mathrm{D}_{\mathrm{qc}}}
\newcommand{\pf}{{\mathrm{perf}}}
\newcommand{\Dperf}{\mathrm{D}^{\pf}}

\theoremstyle{plain}

\newtheorem{theorem}{Theorem}[section]

\newtheorem{lemma}[theorem]{Lemma}
\newtheorem{proposition}[theorem]{Proposition}
\newtheorem{corollary}[theorem]{Corollary}

\theoremstyle{definition}

\newtheorem{definition}[theorem]{Definition}

\newtheorem{example}[theorem]{Example}

\theoremstyle{remark}

\newtheorem{remark}[theorem]{Remark}

\title{Simultaneous categorical resolutions}
\author{Alexander Kuznetsov}
\address{{\sloppy
\parbox{0.99\textwidth}{
Algebraic Geometry Section, Steklov Mathematical Institute of Russian Academy of Sciences,\\
8 Gubkin str., Moscow 119991 Russia
\\[5pt]
Laboratory of Algebraic Geometry, National Research University Higher School of Economics, Russia
}\bigskip}}
\email{akuznet@mi-ras.ru}
\date{}
\thanks{I was partially supported by the HSE University Basic Research Program.}

\dedicatory{To Olivier Debarre, my friend and coauthor}

\begin{document}

\begin{abstract}
We introduce the notion of a simultaneous categorical resolution of singularities, 
a categorical version of simultaneous resolutions of rational double points of surface degenerations.
Furthermore, we suggest a construction of simultaneous categorical resolutions which, in particular, 
applies to the case of a flat projective 1-dimensional family of varieties of arbitrarily high even dimension 
with ordinary double points in the total space and central fiber.
As an ingredient of independent interest, we check that the property of a geometric triangulated category linear over a base 
to be relatively smooth and proper can be verified fiberwise.
As an application we construct a smooth and proper family of~$K3$ categories with general fiber the~$K3$ category of a smooth cubic fourfold
and special fiber the derived category of the~$K3$ surface of degree~6 associated with a nodal cubic fourfold.
\end{abstract}

\maketitle


\section{Introduction}

Let $f \colon X \to B$ be a morphism from a scheme~$X$ to a pointed base scheme $(B,o)$.
We denote by
\begin{equation}
\label{eq:notationn-b-o}
B^o \coloneqq B \setminus \{o\},
\qquad
X^o \coloneqq X \times_B B^o = f^{-1}(B^o),
\qquad\text{and}\qquad 
X_o \coloneqq X \times_B \{o\} = f^{-1}(o)
\end{equation}
the punctured base, its preimage, and the central fiber.

Assume~$B$ is smooth and~$X^o$ is smooth over~$B^o$.
A {\sf simultaneous resolution} of~$(X,X_o)$ is a proper birational morphism~$\pi \colon \tX \to X$ 
such that the morphism 
\begin{equation*}
f \circ \pi \colon \tX \to B 
\end{equation*}
is smooth.
It follows that~$\tX$ and its central fiber~$\tX_o$ are both smooth; 
thus, the morphism~$\pi$ simultaneously resolves singularities of the total space~$X$ and central fiber~$X_o$.
We also usually assume that the morphism~$\pi$ is an isomorphism over~$B^o$.

Simultaneous resolutions are usually hard to construct.
For instance, if one resolves~$X$ by blowing up a subvariety~$Z \subset X$ contained in the central fiber,
the central fiber of the blowup $\Bl_Z(X)$ is the union 
\begin{equation*}
 (\Bl_Z(X))_o = \Bl_Z(X_o) \cup E,
\end{equation*}
where $E$ is the exceptional divisor of the blowup;
in particular, the central fiber is not smooth, and hence the morphism $f \circ \pi$ is not smooth either.
Still, surprisingly, simultaneous resolutions exist in some cases.

\begin{example}[\cite{At}]
\label{example:geometric}
 Let $f \colon X \to B$ be a flat family of complex surfaces over a smooth pointed curve~$(B,o)$
 such that~$X^o$ is smooth over~$B^o$, the total space~$X$ is smooth, 
 and the central fiber~$X_o$ has a single ordinary double point~$x_o \in X_o$.
 Let $B' \to B$ be a double covering branched at~$o$ and let 
 \begin{equation*}
  X' \coloneqq X \times_B B'
 \end{equation*}
 be the base change.
 Now the central fiber~$X'_o \cong X_o$ and the total space~$X'$ both have ordinary double point at~$x_o$.
 Let~$\pi \colon \tX' \to X'$
 be a \emph{small resolution} of~$X'$ (it may or may not be projective, but it always exists as an algebraic space, see~\cite{Artin}).
 It is easy to see that the central fiber of~$\tX'$ is the blowup of~$X'_o$ at~$x_o$, hence is smooth.
 Therefore, the morphism~$\tX' \to B'$ is smooth, and~$\pi$ is a simultaneous resolution of $(X',X'_o)$.
\end{example}

This example is a particular case of the general result of Brieskorn (see also~\cite{Tyurina}) .

\begin{theorem}[\cite{Bri}]
Let $f \colon X \to B$ be a versal deformation of a surface~\mbox{$X_o = f^{-1}(o)$} that has rational double points as singularities.
Then there exists a branched Galois covering~\mbox{$B' \to B$} such that~\mbox{$X' = X \times_B B' \to B'$} 
admits a simultaneous resolution of singularities.
\end{theorem}

This result is very useful for geometry of surfaces.
However, nothing similar is true in higher dimensions, and the reason for this is the lack of small resolutions
(particularly, for ordinary double points).
The goal of this paper is to show that replacing geometric resolutions by categorical resolutions, 
one can preserve some of these results in higher dimensions.

Let~$\kk$ be a base field.
Recall that a pretriangulated differential graded category~$\scD$ is {\sf smooth over~$\kk$} if its diagonal bimodule 
considered as an object of the category of bimodules~$\scD \otimes_\kk \scD^\op$ is perfect.
Similarly, a DG-enhanced triangulated category~$\cD$ is smooth over~$\kk$ if its enhancement~$\scD$ is;
a similar definition works for other types of enhancements,
see, e.g., Definition~\ref{def:sp}, where we use ``geometric enhancement''.
Finally, a triangulated category is {\sf proper} over~$\kk$ if for any its objects~$F_1$, $F_2$ 
the graded vector space~$\Ext^\bullet(F_1,F_2)$ has finite total dimension.

\begin{definition}[{\cite{K08,KL}}]
\label{def:cr}
A {\sf categorical resolution} of a proper $\kk$-scheme~$X$ is a triple~$(\cD,\pi^*,\pi_*)$, 
where
\begin{itemize}
 \item 
 $\cD$ is an enhanced $\kk$-linear triangulated category, and
 \item 
 $\pi^* \colon \Dperf(X) \to \cD$ and $\pi_* \colon \cD \to \Db(X)$ is a pair of enhanced $\kk$-linear triangulated functors,
\end{itemize}
such that 
\begin{enumerate}
 \item 
 $\cD$ is smooth and proper over~$\kk$,
 \item 
 $\pi^*$ is left adjoint to $\pi_*$, 
 i.e., $\Hom(\pi^*F,G) \cong \Hom(F,\pi_*G)$ for all~$F \in \Dperf(X)$, $G \in \cD$, 
 and 
 \item 
 $\pi_* \circ \pi^* \cong \id_{\Dperf(X)}$.
\end{enumerate}
\end{definition}

Of course, if $\pi \colon \tX \to X$ is a (geometric) resolution of singularities the category $\cD = \Db(\tX) \simeq \Dperf(\tX)$ 
together with the derived pullback and pushforward functors provide a categorical resolution of~$X$, 
if~$X$ has rational singularities (see the discussion of this issue in~\cite{KL}).

Examples of categorical resolutions constructed in~\cite{K08} show that this notion is much more flexible than the geometric notion,
for instance one can construct \emph{crepant} categorical resolutions~\cite[Definition~3.4]{K08} 
in the situations when crepant geometric resolutions do not exist.
In particular, one can construct crepant categorical resolutions of ordinary double points in any dimension (see Example~\ref{ex:nodal}).

Extending the definition of a categorical resolution to the setup of simultaneous resolutions,
one naturally arrives at the central definition of this paper.
We use notation~\eqref{eq:notationn-b-o}.

\begin{definition}
\label{def:scr}
Let $f \colon X \to B$ be a flat proper morphism to a pointed base scheme $(B,o)$ with central fiber~$X_o$ such that~$X^o$ is smooth over~$B^o$.
A {\sf simultaneous categorical resolution} of~$(X,X_o)$
is a triple~$(\cD,\pi^*,\pi_*)$, 
where
\begin{itemize}
 \item 
 $\cD$ is an enhanced \emph{$B$-linear} triangulated category, and
 \item 
 $\pi^* \colon \Dperf(X) \to \cD$ and $\pi_* \colon \cD \to \Db(X)$ is a pair of enhanced \emph{$B$-linear} triangulated functors, 
\end{itemize}
such that 
\begin{enumerate}
 \item 
 $\cD$ is smooth and \emph{proper} over~$B$,
 \item 
 $\pi^*$ is left adjoint to~$\pi_*$ (in the same sense as in Definition~\ref{def:cr}),
 and 
 \item 
 $\pi_* \circ \pi^* \cong \id_{\Dperf(X)}$.
\end{enumerate}
\end{definition}

We refer to~\cite[\S2.6]{K06} and~\cite[\S2]{P19} for the definition of $B$-linear categories and functors 
(see also~\S\ref{section:linear-cats} below);
for now just note that the derived category of any $B$-scheme is $B$-linear, 
as well as any its subcategory closed under tensor products with pullbacks of perfect complexes on~$B$.
Similarly, a functor is $B$-linear if it commutes with tensor products by pullbacks of perfect complexes on~$B$.
Note also that if $B = \Spec(R)$ is an affine scheme, a $B$-linear structure is equivalent to an~$R$-linear structure 
in the standard meaning (i.e., $\Hom$-spaces are $R$-modules and the composition law is $R$-linear).

The difference between Definition~\ref{def:scr} and Definition~\ref{def:cr} is, 
first, in $B$-linearity of the resolving category~$\cD$ and functors~$\pi^*$, $\pi_*$, and,
second, in smoothness and properness of~$\cD$ \emph{over~$B$}
(see~\S\ref{subsection:sp} for a discussion of these notions).
As before, a geometric simultaneous resolution provides a categorical one (if~$X$ has rational singularities).

The results of~\cite{K11} (see also~\cite{P19} for a different approach and~\S\ref{subsection:base-change} for a reminder) 
show that given a $B$-linear category~$\cD$ one can construct its base change~$\cD_{B'}$ along any morphism $B' \to B$;
in particular the fiber~$\cD_b$ of~$\cD$ over each point~$b \in B$ is defined.
If~$\cD$ is smooth over~$B$ then~$\cD_b$ is smooth over the residue field of the point~$b$ and if~$\cD$ is proper over~$B$ then~$\cD_b$ is proper.
Thus, in the context of a simultaneous categorical resolution of singularities 
we obtain a family of smooth and proper categories~$\cD_b$ parameterized by the base~$B$.
If, moreover, we assume that the base change of functors~$\pi^*$ and~$\pi_*$ along the inclusion~$B^o \to B$ are equivalences
(this assumption is analogous to the assumption that~$\pi$ is an isomorphism over~$B^o$ which is standard in the geometric context)
then the family of categories~$\cD_b$ agrees with the family of derived categories~$\Db(X_b)$ of fibers of~$f$ over~$B^o$;
in particular, the central fiber~$\cD_o$ of a simultaneous categorical resolution 
provides a \emph{smooth extension} of this family across the point~$o \in B$.

\begin{remark}
It is easy to see that if~$(\cD,\pi^*,\pi_*)$ is a simultaneous categorical resolution of~$(X,X_o)$,
the triple~$(\cD_o,\pi_o^*,\pi_{o*})$ obtained by base change, provides a categorical resolution for~$X_o$.
\end{remark}

\begin{remark}
One can also define a simultaneous categorical resolution for a pair~$(\cC,\cC_o)$,
where~$\cC$ is a~$B$-linear category over a pointed scheme~$(B,o)$
such that~$\cC_{B^o}$ is smooth and proper over~$B^o$ and~$\cC_o$ is the central fiber of~$\cC$.
An example of such situation is provided by Corollary~\ref{cor:cubic} below.
\end{remark}

The main result of the paper is a construction of simultaneous categorical resolutions in a number of new interesting cases.
The most general form of our result is stated in Theorem~\ref{theorem:simultaneous}.
As its formulation is a bit technical, we do not reproduce it here;
instead, we state here its corollary for nodal degenerations which is interesting by itself 
(in Example~\ref{ex:scr} we list some other special situations where the general result can be applied).

\begin{theorem}
\label{theorem:simultaneous-nodal}
Let~$\kk$ be a field of characteristic not equal to~$2$.
Let $f \colon X \to B$ be a flat projective morphism to a smooth pointed $\kk$-curve $(B,o)$ such that~$X^o$ is smooth over~$B^o$.
Assume the total space~$X$ and the central fiber~$X_o$ both have an ordinary double point at~$x_o \in X_o \subset X$ and are smooth elsewhere.
If $\dim(X_o)$ is even, then after a possible quadratic extension of the base field~$\kk$,
there is a simultaneous categorical resolution of~$(X,X_o)$.
\end{theorem}

When applied in the case of relative dimension~2, the constructed categorical resolution is equivalent 
to the geometric simultaneous resolution described in Example~\ref{example:geometric}
(in that case a quadratic extension of the base field may also be necessary to make a small resolution defined).
The assumption that both~$X$ and~$X_o$ have ordinary double points at~$x_o$ plays the same role,
and as in Example~\ref{example:geometric} it can be achieved by the double covering trick, if the original total space~$X$ is smooth at~$x_o$.

To prove Theorem~\ref{theorem:simultaneous-nodal} (and the more general Theorem~\ref{theorem:simultaneous})
we apply the construction of a categorical resolution of singularities from~\cite{K08} (which we remind in Theorem~\ref{theorem:categorical})
to an appropriate blowup~$\pi \colon \tX \to X$; it gives us a $B$-linear triangulated subcategory~$\cD \subset \Db(\tX)$.
It easily follows from the construction that the base change of the category~$\cD$ to the open subscheme~$B^o \subset B$ 
is equivalent to~$\Db(X^o)$;
in particular it is smooth and proper over~$B^o$.
So, the main assertion of the theorem is that~$\cD$ is smooth and proper over a neighborhood of the central point~$o \in B$.

To deduce this we prove a result of independent interest, Theorem~\ref{theorem:sp}, 
saying that if~$\cD$ is a \emph{geometric} $B$-linear category then (under appropriate assumptions) 
the subset~$B_{\mathrm{sm,pr}} \subset B$
of points~$b \in B$ such that~$\cD_b$ is smooth and proper (over the residue field of the point~$b$) is open,
and moreover, the base change category~$\cD_{B_{\mathrm{sm,pr}}}$ is smooth and proper over~$B_{\mathrm{sm,pr}}$. 
Thanks to this result, to prove Theorem~\ref{theorem:simultaneous-nodal}, we just need to check that the central fiber~$\cD_o$
of the constructed categorical resolution~$\cD$ of~$X$ is smooth and proper.

We do this in two steps.
First, we identify (the perfect part of) the category~$\cD_o$ as an explicit subcategory of the central fiber of the family~$\tX \to B$,
which by construction is a reducible scheme with two components, the exceptional divisor~$E \subset \tX$ of the blowup~$\tX \to X$,
and the strict transform~$X'_o \subset \tX$ of the central fiber~$X_o \subset X$.
Then we check that this subcategory is in fact equivalent to an admissible subcategory of one of the components, $X'_o$;
again, this is based on a more general result proved in Proposition~\ref{prop:cc-cc2-equivalence}, which is also interesting by itself.
Since~$X'_o$ is smooth and proper by assumption, it follows that the category~$\cD_o$ is smooth and proper,
hence~$\cD$ is smooth and proper over~$B$ and thus provides a simultaneous categorical resolution for~$(X,X_o)$.
This completes the proof of the theorem.

We expect Theorem~\ref{theorem:simultaneous-nodal} to have many applications, 
similar to those of geometric simultaneous resolutions for surfaces.
To illustrate this we provide one sample application in the case of cubic fourfolds
(see~\S\ref{sec:cubic4} for a reminder about K3 categories and K3 surfaces associated to cubic fourfolds).

\begin{corollary}
\label{cor:cubic}
Let~$\kk$ be an algebraically closed field of characteristic not equal to~$2$.
Let~$Y$ be a cubic fourfold with one ordinary double point.
There exists a smooth pointed curve~$(B,o)$, 
a family of cubic fourfolds~$X \to B$ smooth over~$B^o$ and with central fiber isomorphic to~$Y$, 
and a $B$-linear category~$\cA$ smooth and proper over~$B$ 
such that 
\begin{itemize}
\item 
for $b \ne o$ the fiber~$\cA_b$ is equivalent to the $K3$ category of the smooth cubic fourfold~$X_b$,
\item 
the central fiber~$\cA_o$ is equivalent to the derived category of the smooth $K3$ surface of degree~$6$
that provides a categorical resolution for the~$K3$ category of~$Y$.
\end{itemize}
\end{corollary}

A similar result can be proved for nodal Gushel--Mukai fourfolds; 
in this case the central fiber of the smooth and proper family of categories is equivalent to the $K3$ category of a Verra fourfold, see~\cite{CKKM}.

The result of Corollary~\ref{cor:cubic} shows that the family of $K3$ categories of cubic fourfolds 
a priori defined over the open part~$\fC \setminus (\fC_2 \cup \fC_6)$ 
of the corresponding period domain (here $\fC_d$ stand for the Noether--Lefschetz divisors defined by Hassett)
extends across the general point of the Noether--Lefschetz divisor~$\fC_6$.
Unfortunately, the same construction does not allow us to extend the family across the divisor~$\fC_2$ as well.

It would be very interesting to find generalizations of the constructions of this paper 
that would allow us to construct simultaneous categorical resolutions (or at least smooth extensions of families of categories) 
in more general situations (in particular, to extend the family of~$K3$ categories of cubic fourfolds across the divisor~$\fC_2$).
We hope to come back to this question in the future.

\medskip

{\bf Conventions.}
In this paper we work over an arbitrary field~$\kk$ (imposing restriction on~$\kk$ if necessary).
All schemes are noetherian of finite type over~$\kk$.
For a scheme~$X$ we denote by
\begin{equation*}
\Dperf(X) \subset \Db(X) \subset \Dqc(X)
\end{equation*}
the derived category of perfect complexes, the bounded derived category of coherent sheaves, 
and the unbounded derived category of sheaves of~$\cO_X$-modules with quasicoherent cohomology, respectively.
All pullback, pushforward, and tensor product functors considered in the paper are derived, although we use underived notation for brevity.
For a functor~$\xi$ we usually denote by~$\xi^*$ and~$\xi^!$ its left and right adjoint functors (if they exist), respectively.

\medskip

{\bf Acknowledgements.}
This paper grew out of my attempt to answer several questions asked by Shinnosuke Okawa; I am very grateful to him for asking.
I would also like to thank Sasha Efimov, Alex Perry, Nick Rozenblyum, Evgeny Shinder, Jenia Tevelev, and the anonymous referee
for useful discussions and comments.


\section{Linear categories}
\label{section:linear-cats}

Let~~$f \colon X \to B$ be a morphism of schemes.
A cocomplete subcategory $\cD \subset \Dqc(X)$ is~{\sf $B$-linear}, 
if it is closed with respect to tensor product by pullbacks of perfect complexes on~$B$:
\begin{equation*}
\cD \otimes f^*(\Dperf(B)) \subset \cD.
\end{equation*}
The same definition applies to subcategories of~$\Db(X)$ or~$\Dperf(X)$.

Similarly, a functor $\xi \colon \cD_1 \to \cD_2$ between $B$-linear subcategories in the derived categories 
of $B$-schemes~$f_1 \colon X_1 \to B$ and~$f_2 \colon X_2 \to B$ is {\sf $B$-linear} if it commutes with tensor products, i.e.,
there is a functorial (in both arguments) isomorphism
\begin{equation*}
\xi(F \otimes f_1^*G) \cong \xi(F) \otimes f_2^*G
\end{equation*}
for any $F \in \cD_1$ and $G \in \Dperf(B)$.

\subsection{Base change for admissible subcategories}
\label{subsection:base-change}

In this subsection we recall the construction of base change for semiorthogonal decompositions from~\cite{K11}.
For a subcategory $\cD \subset \Db(X)$ we set 
\begin{equation}
\label{eq:cd-pf}
\cD^\pf = \cD \cap \Dperf(X).
\end{equation}
Similarly, for a subcategory~$\cP \subset \Dperf(X)$ we denote by
\begin{equation*}
\widehat{\cP} \subset \Dqc(X)
\end{equation*}
the minimal cocomplete (i.e., closed under arbitrary direct sums) triangulated subcategory containing~$\cP$.

Now, assume that $X$ is a quasiprojective scheme and let
\begin{equation}
\label{eq:sod-dbx}
\Db(X) = \langle \cD_1, \dots, \cD_n \rangle
\end{equation}
be a semiorthogonal decomposition with admissible components
(\emph{strong} semiorthogonal decomposition in the terminology of~\cite{K11}).
We denote by
\begin{equation*}
\pr_{\cD_i} \colon \Db(X) \to \Db(X)
\end{equation*}
the projection functors to the components~$\cD_i$.
We usually assume that the functors~$\pr_{\cD_i}$ \emph{have finite cohomological amplitude} (see~\cite[\S2.3]{K08}).

\begin{proposition}[{\cite[Proposition~4.1 and Proposition~4.2]{K11}}]
\label{prop:extensions}
If~\eqref{eq:sod-dbx} is a semiorthogonal decomposition with admissible components then there are semiorthogonal decompositions
\begin{equation*}
\Dperf(X) = \langle (\cD_1)^\pf, \dots, (\cD_n)^\pf \rangle
\qquad\text{and}\qquad
\Dqc(X) = \langle \widehat{(\cD_1)^\pf}, \dots, \widehat{(\cD_n)^\pf} \rangle,
\end{equation*}
and if the projection functors of~\eqref{eq:sod-dbx} have finite cohomological amplitude
the second of these decompositions is \emph{compatible} with~\eqref{eq:sod-dbx} in the sense that
\begin{equation}
\label{eq:compatibility}
\cD_i = \widehat{(\cD_i)^\pf} \cap \Db(X).
\end{equation}
In particular, the components of any perfect complex~$F \in \Dperf(X)$ with respect to the above decomposition of~$\Dperf(X)$
coincide with the components of~$F$ considered as an object of~$\Db(X)$ or~$\Dqc(X)$ with respect to the corresponding decompositions.
\end{proposition}

We will say that the above decompositions of~$\Dperf(X)$ and~$\Dqc(X)$ are {\sf induced} by~\eqref{eq:sod-dbx}.

Assume now that~$f \colon X \to B$ is a proper morphism of quasiprojective schemes
and the semiorthogonal decomposition~\eqref{eq:sod-dbx} is~{\sf $B$-linear},
i.e., each component~$\cD_i$ is $B$-linear;
then the projection functors~$\pr_{\cD_i}$ are~$B$-linear by~\cite[Lemma~2.8]{K11}.
Let~$B' \to B$ be a base change from a quasiprojective base scheme~$B'$ such that~$X$ and~$B'$ are~\emph{$\Tor$-independent} over~$B$
(this is automatic if~$f$ is flat).
We consider the fiber product
\begin{equation*}
X' = X \times_B B'
\end{equation*}
and denote by~$\phi \colon X' \to X$ and~$f' \colon X' \to B'$ the induced morphisms of the base change diagram
\begin{equation*}
\xymatrix{
X' \ar[d]_{f'} \ar[r]^\phi & 
X \ar[d]^f 
\\
B' \ar[r]^{\phi_B} &
B.
}
\end{equation*}
When necessary we denote the morphism~$B' \to B$ by~$\phi_B$.

\begin{theorem}[{\cite[Propositions~5.1, 5.3 and Theorem~5.6]{K11}}]
\label{thm:base-change}
Let~$f \colon X \to B$ be a proper morphism of quasiprojective schemes.
Let~\eqref{eq:sod-dbx} be a $B$-linear semiorthogonal decomposition with admissible components and projection functors of finite cohomological amplitude.
Let~$B' \to B$ be a quasiprojective morphism such that~$X$ and~$B'$ are~\emph{$\Tor$-independent} over~$B$.
Set
\begin{enumerate}
\renewcommand{\theenumi}{\roman{enumi}}
\item 
$\cP_{i B'} \coloneqq \left\langle \phi^*\left((\cD_i)^\pf\right) \otimes {f'}^*\Dperf(B') \right\rangle^\oplus \subset \Dperf(X')$, 
where the angle brackets stand for the triangulated envelope and the superscript~$\oplus$ stands for the idempotent completion;
\item 
$\cD_{i B'} \coloneqq \widehat{\cP_{i B'}} \cap \Db(X') \subset \Db(X')$.
\end{enumerate}
Then there are semiorthogonal decompositions
\begin{equation}
\label{eq:sod-db-xp}
\Dperf(X') = \langle \cP_{1 B'}, \dots, \cP_{n B'} \rangle,
\qquad\text{and}\qquad 
\Db(X') = \langle \cD_{1 B'}, \dots, \cD_{n B'} \rangle
\end{equation}
which are compatible, i.e., $(\cD_{i B'})^\pf = \cP_{i B'}$.

Moreover, the components~$\cD_{i B'}$ of the semiorthogonal decomposition~\eqref{eq:sod-db-xp} of~$\Db(X')$ 
are admissible and their projection functors~$\pr_{\cD_{i B'}}$ have finite cohomological amplitude.

Finally, 
the pushforward and pullback functors $\phi_* \colon \Dqc(X') \to \Dqc(X)$ and $\phi^* \colon \Dqc(X) \to \Dqc(X')$
are compatible with the semiorthogonal decompositions of~$\Dqc(X')$ and~$\Dqc(X)$ 
constructed in Proposition~\textup{\ref{prop:extensions}}:
\begin{equation*}
\phi_*\left(\widehat{\cP_{i B'}}\right) \subset \widehat{(\cD_{i})^\pf}
\qquad\text{and}\qquad 
\phi^*\left(\widehat{(\cD_{i})^\pf}\right)
\subset \widehat{\cP_{i B'}}.
\end{equation*}
\end{theorem}

We will use the following obvious consequence of this compatibility.

\begin{corollary}
\label{cor:bc}
Under the assumptions of Theorem~\textup{\ref{thm:base-change}}, assume~$X \to B$ is flat.
Let $F \in \Db(X)$ be an object such that~$\phi^*(F) \in \Db(X')$.
Then 
\begin{equation*}
\phi^*(\pr_{\cD_i}(F)) \cong \pr_{\cD_{i B'}}(\phi^*(F)),
\end{equation*}
where ~$\pr_{\cD_{i B'}}$ is the projection functor of~\eqref{eq:sod-db-xp};
in particular, $\phi^*(\pr_{\cD_i}(F)) \in \Db(X')$.
\end{corollary}

One of the applications of base change technique in~\cite{K11} 
is the existence of Fourier--Mukai kernels for the projection functors~$\pr_{\cD_i}$,
see~\cite[Theorem~7.1]{K11} for the absolute case and~\cite[Theorem~7.3]{K11} for the relative case.
We will usually denote by
\begin{equation}
\label{eq:delta-d}
\Delta_{\cD_i} \in \Db(X \times_B X)
\end{equation}
the corresponding kernels, i.e., the objects such that there is an isomorphism of functors
\begin{equation*}
\pr_{\cD_i} \cong \Phi_{\Delta_{\cD_i}},
\end{equation*}
where the right side is the Fourier--Mukai functor with kernel~$\Delta_{\cD_i}$.
As~\cite[Theorem~7.3]{K11} proves, $\Delta_{\cD_i}$ is just the component of the structure sheaf~$\cO_{\Delta_X}$ 
of the diagonal~$\Delta_X \subset X \times_B X$ with respect to the semiorthogonal decomposition of~$\Db(X \times_B X)$ 
obtained by base change of~\eqref{eq:sod-dbx} along the projection~$X \to B$.
An application of Corollary~\ref{cor:bc} to these objects gives

\begin{corollary}
\label{cor:bc-diagonals}
Under the assumptions of Theorem~\textup{\ref{thm:base-change}}, assume~$X \to B$ is flat.
Then 
\begin{equation*}
\Delta_{\cD_{i B'}} \cong \phi^*\Delta_{\cD_i},
\end{equation*}
where $\phi \colon X'\times_{B'} X' = (X \times_B X) \times_B B' \to X \times_B X$ is the natural projection.
\end{corollary}

\begin{remark}
In~\cite{P19} an alternative approach to base change is developed.
In this approach the basic object is a semiorthogonal decomposition of~$\Dperf(X)$, not of~$\Db(X)$,
and the results for~$\Db(X)$ are extracted from those for~$\Dperf(X)$ by means of the equivalence
\begin{equation*}
\Db(X) \simeq \Fun_{\Dperf(B)}(\Dperf(X)^{\mathrm{op}}, \Db(B)),
\end{equation*}
see~\cite[Theorem~4.26 and~\S4.8]{P19} for more details.
In particular, in the situation of base change the components of the induced semiorthogonal decomposition of~$\Db(X')$ are obtained as
\begin{equation*}
\cD_{i B'} = \Fun_{\Dperf(B')}(\cP_{i B'}^{\mathrm{op}}, \Db(B')).
\end{equation*}
However, we prefer to use the more straightforward approach described above.
\end{remark}

\subsection{Smoothness and properness}
\label{subsection:sp}

Let~$f \colon X \to B$ be a flat proper morphism and let~$\cD \subset \Db(X)$ be an admissible $B$-linear subcategory.
For any pair of objects $F_1,F_2 \in \cD$ we define
\begin{equation*}
\cRHom_B(F_1,F_2) \coloneqq f_*\cRHom(F_1,F_2) \in \Dqc(B).
\end{equation*}
The following definition rephrases in geometric terms the definitions of a category smooth and proper over base.
For intrinsic categorical definitions we refer to~\cite[Definition~4.5 and~Lemma~4.7]{P19}.

\begin{definition}
\label{def:sp}
Let~$f \colon X \to B$ be a flat proper morphism,
let~$\cD \subset \Db(X)$ be an admissible $B$-linear subcategory with an admissible orthogonal,
and set~$\cP \coloneqq \cD^\pf$.
We will say that
\begin{itemize}
\item 
the category~$\cP$ is {\sf proper over~$B$} if $\cRHom_B(F_1,F_2) \in \Dperf(B)$ for any~$F_1,F_2 \in \cP$;
\item 
the category~$\cP$ is {\sf smooth over~$B$} 
if the projection functor of~$\cD$ has finite cohomological amplitude 
and the corresponding Fourier--Mukai kernel~$\Delta_\cD \in \Dperf(X \times_B X)$ is perfect.
\end{itemize}
\end{definition}

Recall from~\S\ref{subsection:base-change} that the category~$\cD$ can be reconstructed from~$\cP = \cD^\pf$;
in particular, although in the definitions of smoothness we use~$\cD$, this is still a property of~$\cP$.
Let us also briefly explain why the above ad hoc definition is equivalent to the one from~\cite{P19}.
By~\cite[Lemma~4.7]{P19} for properness there is nothing to explain, 
and for smoothness note that the category of continuous~$B$-linear functors from~$\widehat\cP$ to itself 
can be identified with an appropriate subcategory of~$\Dqc(X \times_B X)$ in such a way that the identity functor corresponds to~$\Delta_\cD$;
therefore the condition~$\Delta_\cD \in \Dperf(X \times_B X)$ is equivalent to the condition 
that the identity functor is a compact object of the functor category.

We have the following obvious observations.

\begin{lemma}
\label{lemma:sp-geometric}
Let $f \colon X \to B$ be a flat proper morphism.
Let~$\cD \subset \Db(X)$ be an admissible $B$-linear subcategory with 
an admissible orthogonal and the corresponding projection functors of finite cohomological amplitude.
If~$X$ is smooth or proper over~$B$ then~$\cD^\pf$ is smooth or proper over~$B$, respectively.
\end{lemma}
\begin{proof}
If~$X$ is smooth over~$B$ then the structure sheaf of the diagonal~$\Delta_X  \subset X \times_B X$ is perfect.
By Proposition~\ref{prop:extensions} (applied to~$X \times_B X$) its component~$\Delta_\cD$ is also perfect, hence~$\cD^\pf$ is smooth over~$B$.

Similarly, assume that~$X$ is proper over~$B$.
If~$F_1,F_2 \in \cD^\pf$ then~$\cRHom(F_1,F_2) \in \Dperf(X)$, and since~$f$ is flat and proper, 
the object~$\cRHom_B(F_1,F_2)$ is also perfect by~\cite[III.4.8]{SGA6}.
Therefore~$\cD^\pf$ is proper over~$B$.
\end{proof}

\begin{lemma}
\label{lemma:sp-obvious}
Smoothness and properness over a base are preserved under base change.
Moreover, smoothness and properness over a base are local \textup(over the base\textup) properties.
\end{lemma}

\begin{proof}
Properness of the base change category follows from the isomorphism 
\begin{equation*}
\phi_B^*\cRHom_B(F_1,F_2) \cong \cRHom_{B'}(\phi^*F_1,\phi^*F_2),
\qquad 
F_1,F_2 \in \Dperf(X).
\end{equation*}
Smoothness of the base change category and the second part of the lemma follow from Corollary~\ref{cor:bc-diagonals} and the local nature of perfectness.
\end{proof}

The following result will be used in the proofs of Theorem~\ref{theorem:sp} and Theorem~\ref{theorem:simultaneous}.

\begin{proposition}
\label{prop:perf-coh}
Let $X$ be a projective scheme over a field~$\kk$.
Let $\cD \subset \Db(X)$ be an admissible subcategory such that~${}^\perp\cD$ is admissible.
If the category~$\cD^\pf$ is smooth and proper over~$\kk$ then~$\cD = \cD^\pf$.
\end{proposition}

\begin{proof}
Since~$\cD^\pf$ is smooth over~$\kk$, it has a strong generator by~\cite[Lemma~3.5, 3.6]{Lunts10}; in other words, it is \emph{regular}.
It is also idempotent complete, because~$\cD$ and~$\Dperf(X)$ are, and proper over~$\kk$ by assumption.
Therefore, $\cD^\pf$ is \emph{saturated} by~\cite[Theorem~1.3]{BV}; 
in other words every contravariant triangulated functor $\cD^\pf \to \Db(\kk)$ is representable.
Applying~\cite[Proposition~2.6]{BK}, we conclude that
there is a semiorthogonal decomposition
\begin{equation*}
\cD = \langle \cD', \cD^\pf \rangle.
\end{equation*}
On the other hand, the semiorthogonal decomposition~$\Db(X) = \langle \cD, {}^\perp\cD \rangle$ gives decompositions
\begin{equation*}
\Db(X) = \langle \cD', \cD^\pf, {}^\perp\cD \rangle
\qquad\text{and}\qquad
\Dperf(X) = \langle \cD^\pf, ({}^\perp\cD)^\pf \rangle
\end{equation*}
(to construct the second we use Proposition~\ref{prop:extensions}),
which imply that $\cD' \subset (\Dperf(X))^\perp = 0$, hence~$\cD' = 0$ and so~$\cD = \cD^\pf$.
\end{proof}

The main result of this section is the following. 

\begin{theorem}
\label{theorem:sp}
Let~$f \colon X \to B$ be a flat proper morphism of quasiprojective schemes 
and let~\mbox{$\cD \subset \Db(X)$} be an admissible $B$-linear subcategory
such that~${}^\perp\cD$ is admissible and the projection functors 
of the semiorthogonal decomposition~$\Db(X) = \langle \cD, {}^\perp\cD \rangle$ have finite cohomological amplitude.
Denote by~\mbox{$B_{\mathrm{sm}} \subset B$} and $B_{\mathrm{sm,pr}} \subset B$ the sets of points $b \in B$ 
such that the category~$(\cD_b)^\pf$ is smooth 
and such that~$(\cD_b)^\pf$ is smooth and proper over the residue field of~$b$, respectively.

Then the subsets~$B_{\mathrm{sm}}$ and~$B_{\mathrm{sm,pr}}$ are open in~$B$,
the category~$(\cD_{B_{\mathrm{sm}}})^\pf$ is smooth over~$B_{\mathrm{sm}}$ and
the category~$(\cD_{B_{\mathrm{sm,pr}}})^\pf$ is smooth and proper over~$B_{\mathrm{sm,pr}}$.

Moreover, if the scheme~$B_{\mathrm{sm,pr}}$ is regular, there is an equality of categories
\begin{equation}
\label{eq:cd-b-cd-pf}
(\cD_{B_{\mathrm{sm,pr}}})^\pf = \cD_{B_{\mathrm{sm,pr}}}.
\end{equation} 
\end{theorem}

It is an interesting question 
if a similar result holds for $B$-linear categories 
which cannot be represented as admissible subcategories in derived categories of $B$-linear varieties ---
we are not aware of any results in this direction.

The proof of Theorem~\ref{theorem:sp} is based on the following results about perfect complexes on schemes
which should be well-known but we did not find a reference.

\begin{lemma}
\label{lemma:general}
Let $i \colon Z \to Y$ be a closed embedding of noetherian schemes.
If $F \in \Db(Y)$ is such that~$i^*F \in \Dperf(Z)$ then there is an open subset~$U \subset Y$ containing~$Z$ such that~$F\vert_U$ is perfect.
\end{lemma}
\begin{proof}
First, assume $Y = \Spec(A)$ is affine.
Let $z \in Z$ be a closed point, 
let $\fm \subset A$ be the corresponding maximal ideal,
and let~$A_\fm$ be the localization of~$A$ at~$\fm$.

First, we check that $F_\fm \coloneqq F \otimes_A A_\fm$ is perfect.
Indeed, let~$P^\bullet$ be the minimal free bounded above resolution of~$F_\fm$.
Its terms have finite rank because~$F \in \Db(Y)$ and~$Y$ is noetherian.
Recall that the minimality of~$P^\bullet$ means that the complex~$P^\bullet \otimes_{A_\fm} (A_\fm/\fm A_\fm)$ has zero differentials.
Since on the other hand the complex~$P^\bullet \otimes_{A_\fm} (A_\fm/\fm A_\fm)$ is quasiisomorphic 
to the localization of the perfect complex~$i^*F$, 
it is perfect, and since this is a complex with zero differential, it is bounded.
Therefore, $P^\bullet$ is also bounded, hence it is a perfect complex, hence the same is true for~$F_\fm$.

Let again $P^\bullet$ be a bounded complex of free~$A_\fm$-modules of finite rank quasiisomorphic to~$F_\fm$.
Note that the quasiisomorphism can be represented by an actual morphism of complexes~$P^\bullet \to F_\fm$.
The components of this morphism and the differentials of~$P^\bullet$ comprise a finite number of matrices with elements 
in the localization~$A_\fm$ of the ring and localizations~$(F^i)_\fm$ of terms of the complex~$F$; 
they contain a finite number of denominators which are not contained in~$\fm$.
Therefore, there is an open subscheme~$U_z \subset Y$ containing~$z$, 
a bounded complex $\tilde P^\bullet$ of free~$\cO_{U_z}$-modules of finite rank,
and a morphism $\tilde P^\bullet \to F\vert_{U_z}$ 
(the differentials of~$\tilde P^\bullet$  and the components of the morphism are given by the same matrices as before)
which becomes an isomorphism after tensor product with~$A_\fm$.
The cone 
\begin{equation*}
G \coloneqq \Cone(\tilde P^\bullet \to F\vert_{U_z})
\end{equation*}
is, therefore, an object of~$\Db(U_z)$ which becomes zero after tensoring with~$A_\fm$.
Thus, it is zero on a smaller neighborhood~$U'_z \subset U_z$ of~$z$, 
and hence $F\vert_{U'_z} \cong \tilde P^\bullet\vert_{U'_z}$ is perfect.

Now let $Y$ be arbitrary.
The above argument shows that each point~$z \in Z$ has a neighborhood~$U'_z \subset Y$ containing~$z$ such that~$F\vert_{U'_z}$ is perfect.
Therefore, $F$ is perfect on the open subset~$U \coloneqq \bigcup_{z \in Z} U'_z$ of~$Y$ containing~$Z$.
\end{proof}

\begin{corollary}
\label{cor:general}
Let $Y \to B$ be a proper morphism.
For a closed point $b \in B$ denote by~$i \colon Y_b \hookrightarrow Y$ the embedding of the fiber over~$b$.
If $F \in \Db(Y)$ is an object such that~$i^*F \in \Dperf(Y_b)$
then there is an open subset $U \subset B$ containing~$b$ such that~$F\vert_{Y_U} \in \Dperf(Y_U)$.
\end{corollary}
\begin{proof}
By Lemma~\ref{lemma:general} there is an open subset~$V \subset Y$ containing~$Y_b$ such that~$F\vert_V$ is perfect.
Since the morphism~$Y \to B$ is proper, the image of~$Y \setminus V$ in~$B$ is closed.
Since it does not contain the point~$b$, we can just take~$U$ to be its complement in~$B$.
\end{proof}

Now we are ready to prove the theorem.

\begin{proof}[Proof of Theorem~\textup{\ref{theorem:sp}}]
Let~$\Delta_\cD \in \Db(X \times_B X)$ be the Fourier--Mukai kernel of the projection functor~$\pr_\cD$ defined by~\eqref{eq:delta-d}.
Similarly, for each~$b \in B$ let~$\Delta_{\cD_b} \in \Db(X_b \times X_b)$ be the analogous object 
for the category~$\cD_b \subset \Db(X_b)$.
By Corollary~\ref{cor:bc-diagonals} we have
\begin{equation*}
\Delta_{\cD_b} \cong i^*\Delta_\cD,
\end{equation*}
where $i \colon X_b \times X_b \hookrightarrow X \times_B X$ is the embedding of the fiber of~$X \times_B X \to B$ over~$b$.
If~$b \in B_{\mathrm{sm}}$ then~$\Delta_{\cD_b}$ is perfect by Definition~\ref{def:sp}.
Therefore, by Corollary~\ref{cor:general} there is an open subset~$U \subset B$ containing~$b$
such that~$\Delta_\cD\vert_{X_U \times_U X_U}$ is also perfect.
But 
\begin{equation*}
\Delta_\cD\vert_{X_U \times_U X_U} \cong \Delta_{\cD_U},
\end{equation*}
again by Corollary~\ref{cor:bc-diagonals},
hence~$(\cD_U)^\pf$ is smooth over~$U$, and hence $U \subset B_{\mathrm{sm}}$.
Thus, $B_{\mathrm{sm}}$ is open.
Moreover, by the second part of Lemma~\ref{lemma:sp-obvious} the category~$(\cD_{B_{\mathrm{sm}}})^\pf$ is smooth over~$B_{\mathrm{sm}}$.

For the second part of the theorem we may assume that $B = B_{\mathrm{sm}}$, so that~$\cD^\pf$ is smooth over~$B$.
By Definition~\ref{def:sp} the object~$\Delta_\cD \in \Db(X \times_B X)$ is perfect, 
and since the category of perfect complexes of a fiber product is generated 
by tensor products of perfect complexes pulled back from the factors (\cite[Lemma~5.2]{K11}),
there is a finite number of perfect complexes~$G'_i,G''_i \in \Dperf(X)$ such that
\begin{equation*}
\Delta_\cD \in \langle p_1^*G'_i \otimes p_2^*G''_i \rangle^\oplus,
\end{equation*}
where $p_1,p_2 \colon X \times_B X \to X$ are the projections 
and as before the superscript~$\oplus$ means the idempotent completion of the subcategory generated by objects in the angle brackets.
Since the Fourier--Mukai functor with kernel~$\Delta_\cD$ is isomorphic to the projection functor~$\pr_\cD$,
it follows that for any~$F \in \cD^\pf$ we have 
\begin{equation*}
F \cong 
p_{2*}(p_1^*F \otimes \Delta_\cD) \in 
\langle p_{2*}(p_1^*F \otimes p_1^*G'_i \otimes p_2^*G''_i) \rangle^\oplus =
\langle p_{2*}p_1^*(F \otimes G'_i) \otimes G''_i \rangle^\oplus =
\langle f^*f_*(F \otimes G'_i) \otimes G''_i \rangle^\oplus
\end{equation*}
(the first equality is the projection formula and the second is base change).
Since $f_*(\Dperf(X)) \subset \Dperf(B)$ (see~\cite[III.4.8]{SGA6})
it follows that the finite number of perfect objects~$G''_i$ generate~$\cD^\pf$ over~$B$ up to idempotent completion.
Therefore, for any base change~$B' \to B$ the pullbacks of~$G''_i$ to~$X' = X \times_B B'$ 
generate~$(\cD_{B'})^\pf$ over~$B'$ up to idempotent completion.
It follows that properness of~$(\cD_{B'})^\pf$ over~$B'$ is equivalent to perfectness 
of the finite number of objects~$\cRHom_{B'}(\phi^*G''_i,\phi^*G''_j)$,
where~$\phi \colon X' \to X$ is the base change morphism.

Since~$X$ is flat over~$B$, we have~$\cRHom_{B'}(\phi^*G''_i,\phi^*G''_j) \cong \phi_B^*\cRHom_B(G''_i,G''_j)$,
where~$\phi_B \colon B' \to B$ is the base change morphism,
so, Lemma~\ref{lemma:general} shows that if these objects are perfect for~$B' = \{b\}$, they are also perfect if~$B'$ is a small neighborhood of~$b$.
Therefore, the subset $B_{\mathrm{sm,pr}} \subset B$ is open.
It also follows from the second part of Lemma~\ref{lemma:sp-obvious}
that~$(\cD_{B_{\mathrm{sm,pr}}})^\pf$ is smooth and proper over~$B_{\mathrm{sm,pr}}$.

For the last statement we may assume that~$\cD^\pf$ is smooth and proper over~$B$ and~$B$ is regular.
Let~\mbox{$F \in \cD$}.
For any point $b \in B$ the morphism~$\{b\} \to B$ has finite~$\Tor$-dimension (because~$B$ is regular), 
hence the same is true for the embedding of the fiber~$i \colon X_b \to X$, hence~$i^*F \in \cD_b$.
On the other hand, since the category~$(\cD_b)^\pf$ is smooth and proper over the residue field of~$b$, 
it follows from Proposition~\ref{prop:perf-coh} that~$\cD_b = (\cD_b)^\pf$, hence~$i^*F$ is perfect.
Then by Lemma~\ref{lemma:general} the object~$F$ is perfect in a neighborhood of~$X_b$.
Since this is true for any~$b \in B$, we conclude that~$F$ is perfect, hence~$\cD \subset \Dperf(X)$ and so~$\cD^\pf = \cD$.
\end{proof}

\section{Construction of simultaneous categorical resolutions}

The goal of this section is to construct 
a simultaneous categorical resolution of~$(X,X_o)$ under appropriate assumptions.

In~\S\ref{subsection:perfect-base-change} we recall from~\cite{K08} the construction of 
a categorical resolution~$\cD$ of the total space~$X$ 
(see Theorem~\ref{theorem:categorical} below),
check that it is a~$B$-linear category, and compute the prefect part of its central fiber~$\cD_o$ 
as a subcategory of the perfect derived category of the reducible central fiber of the blowup~$\tX$ of~$X$.

In~\S\ref{subsection:reducible} we show that under appropriate assumptions a subcategory of the derived category of a reducible scheme
can be identified with a subcategory of one of its components.

In~\S\ref{subsection:proof} we impose extra conditions on the categorical resolution~$\cD$ 
constructed in~\S\ref{subsection:perfect-base-change} which imply that its central fiber~$\cD_o$ is smooth and proper 
and that~$\cD$ itself is smooth and proper over~$B$;
in particular we state and prove our main result, Theorem~\ref{theorem:simultaneous}.
We also list in Example~\ref{ex:scr} several cases where Theorem~\ref{theorem:simultaneous} can be applied.

Finally, in~\S\ref{subsection:examples} we prove Theorem~\ref{theorem:nodal-general}, 
a more general and precise version of Theorem~\ref{theorem:simultaneous-nodal} from the Introduction,
and show that it provides an example of a categorical flop.

\subsection{Categorical resolution and its central fiber}
\label{subsection:perfect-base-change}

First, we recall the setup we are working in and introduce some notation.
Let~$\kk$ be an arbitrary field.
Let~$B$ be a smooth curve over~$\kk$ with a (closed) base point $o \in B$, 
let~$f \colon X \to B$ be a flat projective morphism with central fiber $X_o$
such that the morphism~$X^o \to B^o$ is smooth (we use notation~\eqref{eq:notationn-b-o}).
Let~$Z \subset X_o$ be a smooth closed subscheme in the central fiber 
such that both the total space~$X$ and the central fiber~$X_o$ are smooth away from~$Z$. 
Let 
\begin{equation*}
\pi \colon \tX \coloneqq \Bl_{Z}(X) \to X 
\end{equation*}
be the blowup of~$X$ at~$Z$.
We denote by~$E$ the exceptional divisor of~$\pi$ and by~$\eps \colon E \hookrightarrow \tX$ its embedding.
We have the blowup diagram
\begin{equation*}
\xymatrix{
E \ar[r]^\eps \ar[d]_p &
\tX \ar[d]^\pi
\\
Z \ar[r] &
X
}
\end{equation*}
Finally, we denote by 
\begin{equation}
\label{def:l}
L \coloneqq \cO_E(-E)
\end{equation}
the conormal bundle of~$E$ in~$\tX$.

The construction of categorical resolution in~\cite{K08} starts with a choice of
a $Z$-linear left \emph{Lefschetz decomposition} of~$\Db(E)$ with respect to~$L$, 
i.e., a semiorthogonal decomposition of the form
\begin{equation}
\label{eq:Lefschetz-e}
\Db(E) = \langle \cA_{1-m} \otimes L^{1-m}, \dots, \cA_{-1} \otimes L^{-1}, \cA_0 \rangle,
\end{equation}
where the~$\cA_i$ form a chain of $Z$-linear subcategories
\begin{equation*}
 0 \subset \cA_{1-m} \subset \dots \subset \cA_{-1} \subset \cA_0 \subset \Db(E).
\end{equation*}
It is assumed additionally that 
\begin{equation}
\label{eq:co-e-assumption}
p^*(\Dperf(Z)) \subset \cA_0.
\end{equation}

Then the following result is proved (recall Definition~\ref{def:cr}).

\begin{theorem}[{\cite{K08}}]
\label{theorem:categorical}
Assume $X$ is a quasiprojective scheme with rational singularities, 
$Z \subset X$ is a closed subscheme such that the blowup~\mbox{$\tX = \Bl_{Z}(X)$} is smooth,
and its exceptional divisor~$E$ is endowed with a $Z$-linear left Lefschetz decomposition~\eqref{eq:Lefschetz-e} 
with respect to the conormal bundle~$L$
such that every component~$\cA_k$ of~\eqref{eq:Lefschetz-e} is admissible in~$\Db(E)$ and~\eqref{eq:co-e-assumption} holds.
Set 
\begin{equation}
\label{eq:def-cd}
\cD \coloneqq \{F \in \Db(\tX) \mid \eps^*F \in \cA_0 \} \subset \Db(\tX).
\end{equation} 
Then $\pi^*(\Dperf(X)) \subset \cD$ and the triple $(\cD,\pi^*,\pi_*)$ provides a categorical resolution of~$X$.

Furthermore, the functor~$\eps_* \colon \Db(E) \to \Db(\tX)$ is fully faithful 
on the subcategories $\cA_k \subset \Db(E)$ for~$1-m \le k \le -1$,
these categories have the property $(\cA_k)^\pf = \cA_k$,
and there is an $X$-linear semiorthogonal decomposition
\begin{equation}
\label{eq:db-tx}
\Db(\tX) = \langle \eps_*(\cA_{1-m} \otimes L^{1-m}), \dots, \eps_*(\cA_{-1} \otimes L^{-1}), \cD \rangle
\end{equation}
with admissible components and projection functors of finite cohomological amplitude.

Finally, if~$X$ is proper then~$\cD^\pf = \cD$ and it is smooth and proper over~$\kk$.
\end{theorem}

\begin{proof}
Most part of this is proved in~\cite[Theorem~4.4]{K08}.
Apart from that, the equality~\mbox{$(\cA_k)^\pf = \cA_k$} is established in~\cite[Proposition~4.1]{K08},
$X$-linearity of~$\eps_*(\cA_k \otimes L^k)$ follows easily from~$Z$-linearity of~$\cA_k$,
and~$X$-linearity of~$\cD$ is explained in~\cite[Lemma~4.6]{K08}.

Admissibility of the components~$\eps_*(\cA_k \otimes L^k)$ follows  
from the existence of both adjoints to~$\eps_*$ and admissibility of~$\cA_k$.
To prove admissibility of~$\cD$ note that admissibility of~$\cA_k$ 
implies the existence of a right Lefschetz decomposition (see~\cite[Lemma~2.19]{K08})
\begin{equation*}
\Db(E) = \langle \cA_0, \cA_1 \otimes L, \dots, \cA_{m-1} \otimes L^{m-1} \rangle
\end{equation*}
and then a similar argument (see~\cite[Proposition~3.11]{KP21}) gives the semiorthogonal decomposition 
\begin{equation*}
\Db(\tX) = \langle \cD, \eps_!(\cA_1 \otimes L), \dots, \eps_!(\cA_{m-1} \otimes L^{m-1}) \rangle,
\end{equation*}
where $\eps_!(\cF) = \eps_*\cF \otimes \cO_{\tX}(E)[-1]$ is the left adjoint functor of~$\eps^*$.
Now the above decomposition of~$\Db(\tX)$ shows that~$\cD$ is left admissible, 
while~\eqref{eq:db-tx} shows that it is right admissible.
The projection functors of the components of~\eqref{eq:db-tx} have finite cohomological amplitude by~\cite[Proposition~2.5]{K08}.

Finally, $\cD^\pf = \cD$ because~$\Dperf(\tX) = \Db(\tX)$ as~$\tX$ is smooth,
and if~$X$ is proper then~$\tX$ is smooth and proper and~$\cD^\pf$ is smooth and proper by Lemma~\ref{lemma:sp-geometric}.
\end{proof}

The simplest case where the theorem applies is the following.

\begin{example}
\label{ex:nodal}
Assume the base field is algebraically closed of characteristic not equal to~$2$.
Assume~$X$ has an ordinary double point at~$x_0$ and is smooth elsewhere.
Set~$Z = \{x_0\}$.
Then the blowup~\mbox{$\tX = \Bl_Z(X) = \Bl_{x_0}(X)$} is smooth, its exceptional divisor~$E$ is a smooth quadric, 
and the conormal line bundle~$L = \cO_E(-E)$ is the hyperplane line bundle of the quadric.
We consider the standard Lefschetz decomposition of a quadric, see~\cite[\S2.2]{KP19} defined by
\begin{equation*}
\cA_0 = \langle \cS, \cO \rangle,
\end{equation*}
where~$\cS$ is a spinor bundle on~$E$.
Then 
\begin{equation*}
\cA_i = 
\begin{cases}
\cA_0, & \text{\hbox to 1em {\hfill if} $i = -1$ and~$\dim(E)$ is even},\\
\langle \cO \rangle, & \text{\hbox to 1em {\hfill if} $-\dim(E) \le i \le -2$ and~$\dim(E)$ is even},\\ 
& \text{\hbox to 1em {\hfill or} $-\dim(E) \le i \le -1$ and~$\dim(E)$ is odd,}
\end{cases}
\end{equation*}
and is zero otherwise, see~\cite[Lemma~2.4]{KP19}.
Applying Theorem~\ref{theorem:categorical} we obtain a categorical resolution of~$X$ by the category
\begin{equation*}
\cD = \{ F \in \Db(\tX) \mid \eps^*F \in \langle \cS, \cO \rangle \}
\end{equation*}
and the semiorthogonal decomposition~\eqref{eq:db-tx}.
\end{example}

Now assume that $X$ is a scheme over~$B$, hence the same is true for the scheme~$\tX$.
Furthermore, the subcategory~$\cD \subset \Db(\tX)$ is $X$-linear, hence it is a fortiori~$B$-linear.
It makes sense, therefore, to consider its base change~$\cD_b$ to various points $b \in B$.
For $b \ne o$ a description of~$\cD_b$ is straightforward.

\begin{lemma}
\label{lemma:cd-b}
If $b \ne o$ then $\cD_b \simeq \Db(\tX_b) \simeq \Db(X_b)$.
\end{lemma}

\begin{proof}
By Theorem~\ref{thm:base-change} we have a semiorthogonal decomposition
\begin{equation*}
\Db(\tX_b) = \langle (\eps_*(\cA_{1-m} \otimes L^{1-m}))_b, \dots, (\eps_*(\cA_{-1} \otimes L^{-1}))_b, \cD_b \rangle. 
\end{equation*}
Now we apply~\cite[Theorem~6.4]{K11} that shows that each of the first $m-1$ components 
is a subcategory of the base change of the divisor~$E \subset \tX$ along the embedding~$b \hookrightarrow B$.
But~$E$ is supported over the point~$o$, therefore $E_b = \varnothing$ and so each of these components is zero
(this can be also deduced directly from the construction of base change without invoking the above theorem).
Thus, the last component~$\cD_b$ is equal to the entire category~$\Db(\tX_b)$
which is equal to~$\Db(X_b)$ because $\tX_b = X_b$ 
(since the center~$Z$ of the blowup~$\pi$ is contained in the central fiber~$X_o$ of~$f$ and so~$\pi$ is an isomorphism over~$B^o$).
\end{proof}

It is much trickier to describe the central fiber~$\cD_o$.
To do this we introduce more notation and impose an additional assumption:
we assume that the scheme central fiber of~$\tX$, which is a reducible scheme
\begin{equation}
\label{eq:tx-o}
\tX_o = X'_o \cup E,
\end{equation}
\emph{has reduced components}, where
\begin{equation*}
X'_o \cong \Bl_{Z}(X_o)
\end{equation*}
is the strict transform of the central fiber of~$X$.
We denote by 
\begin{equation*}
E_o \coloneqq X'_o \cap E
\end{equation*}
the exceptional divisor of~\mbox{$X'_o \to X_o$};
then we have the following commutative diagram
\begin{equation}
\label{diagram:divisor}
\vcenter{\xymatrix{
E_o \ar[r]^{i_E} \ar[d]_{i_X} &
E \ar[d]_{r_E} \ar[dr]^{\eps}
\\
X'_o \ar[r]_-{r_X} &
X'_o \cup E \ar[r]_-j &
\tX.
}}
\end{equation}

We start with an obvious observation.

\begin{lemma}
\label{lemma:e0-e}
The normal bundle of the central fiber~$\tX_o$ in~$\tX$ is trivial.
Moreover, if the central fiber~\eqref{eq:tx-o} is reduced then the normal bundles of its components are
\begin{equation*}
\cN_{X'_o/\tX} \cong \cO_{X'_o}(-E_o)
\qquad\text{and}\qquad 
\cN_{E/\tX} \cong \cO_{E}(-E_o).
\end{equation*}
In particular, the line bundle~\eqref{def:l} has the form~$L \cong \cO_E(E_o)$.
\end{lemma}

\begin{proof}
The normal bundle of~$\tX_o$ is trivial because it is a fiber of a flat morphism to a smooth curve.
Therefore, we have the isomorphisms
\begin{equation*}
\cO_{X'_o} \cong
\cN_{\tX_o/\tX}\vert_{X'_o} \cong
\cO_{\tX_o}(E + X'_o)\vert_{X'_o} \cong
\cO_{X'_o}(E) \otimes \cO_{{X'_o}}(X'_o) \cong 
\cO_{X'_o}(E_o) \otimes \cN_{X'_o/\tX}
\end{equation*}
(the second isomorphism holds because the central fiber is reduced)
which gives the required isomorphism for~$\cN_{X'_o/\tX}$.
The isomorphism for~$\cN_{E/\tX}$ is proved similarly.
\end{proof}

The next proposition describes the perfect part of the category~$\cD_o$.
We use notation~\eqref{eq:cd-pf} and~\eqref{diagram:divisor}.

\begin{proposition}
\label{prop:cc-perf-b}
Assume the central fiber~\eqref{eq:tx-o} of~$\tX \to B$ is reduced.
If~$\cD \subset \Db(\tX)$ is the subcategory defined by~\eqref{eq:def-cd}
then the perfect part of its central fiber~$(\cD_o)^\pf \subset \Dperf(\tX_o \cup E)$ can be described as
\begin{equation}
\label{eq:ccb-characterization}
(\cD_o)^\pf = \{ F \in \Dperf(X'_o \cup E) \mid r_E^*F \in \cA_0 \}.
\end{equation}
\end{proposition}

\begin{proof}
First, if $F \in \cD^\pf$ then by commutativity of~\eqref{diagram:divisor} and the definition of~$\cD$ we have
\begin{equation*}
r_E^*j^*F \cong \eps^*F \in \cA_0,
\end{equation*}
hence $j^*F$ belongs to the right side of~\eqref{eq:ccb-characterization}.
By identification~$(\cD_o)^\pf = \cP_o$ from Theorem~\ref{thm:base-change} this
implies that the entire category~$(\cD_o)^\pf$ is contained in the right side of~\eqref{eq:ccb-characterization}.

Furthermore, the semiorthogonal decomposition~\eqref{eq:db-tx} combined with Theorem~\ref{thm:base-change} implies that 
the orthogonal $((\cD_o)^\pf)^\perp$ in~$\Dperf(\tX_o)$ is generated by objects of the form~$j^*\eps_*G$,
where~$G \in \cA_k \otimes L^k$ for~$k \le -1$.
So, it remains to check that if~$F$ belongs to the right side of~\eqref{eq:ccb-characterization}
and~$G$ is one of objects as above then~$\Ext^\bullet(F,j^*\eps_*G) = 0$.
By~\eqref{diagram:divisor} we have~$j^*\eps_*G = j^*j_*r_{E*}G$ and since~$j$
is the embedding of a Cartier divisor with trivial normal bundle, we have a distinguished triangle
\begin{equation*}
j^*\eps_*G \to r_{E*}G \to r_{E*}G[2].
\end{equation*}
So, it is enough to check that $\Ext^\bullet(F,r_{E*}G) = 0$, which follows immediately from
the adjunction isomorphism~$\Ext^\bullet(F,r_{E*}G) \cong \Ext^\bullet(r_E^*F,G)$ and~\eqref{eq:Lefschetz-e}.
\end{proof}

\subsection{Perfect derived category of a reducible scheme}
\label{subsection:reducible}

In this subsection we prove Proposition~\ref{prop:cc-cc2-equivalence} that allows us 
to identify some subcategories of the perfect derived category
of a reducible scheme with subcategories of its irreducible component.
This result is crucial for the proof of our main Theorem~\ref{theorem:simultaneous} stated later,
where it will be applied to the category~$(\cD_o)^\pf$ described in Proposition~\ref{prop:cc-perf-b}.

Consider a reducible scheme 
\begin{equation*}
D = D_1 \cup D_2,
\end{equation*}
and consider the scheme intersection of its components
\begin{equation*}
D_1 \cap D_2 = D_0.
\end{equation*}
Then we have a commutative square of schemes
\begin{equation}
\label{diagram:reducible}
\vcenter{\xymatrix{
D_0 \ar[r]^{i_1} \ar[d]_{i_2} \ar[dr]^{r_0} &
D_1 \ar[d]^{r_1}
\\
D_2 \ar[r]_{r_2} &
D
}}
\end{equation}

In the last part of the following lemma we assume the category~$\cT$ and the functors~$\varphi_i$ to be enhanced;
here one can use a DG-enhancement as in Definitions~\ref{def:cr} and~\ref{def:scr} or an enhancement of any other type.

\begin{lemma}
\label{lemma:reducible}
Consider a reducible scheme~$D$ and the diagram~\eqref{diagram:reducible}.
\begin{enumerate}\renewcommand{\theenumi}{\roman{enumi}}
\item 
\label{item:triangle-object}
For any object~$G \in \Dqc(D)$ there is a canonical distinguished triangle
\begin{equation}
\label{eq:triangle-object}
G \to r_{1*}r_1^*G \oplus r_{2*}r_2^*G \to r_{0*}r_0^*G.
\end{equation} 
Its formation commutes with arbitrary direct sums in~$\Dqc(D)$.
\item 
\label{item:triangle-ext}
For any objects~$F,G \in \Dqc(D)$ there is a canonical distinguished triangle 
\begin{equation}
\label{eq:triangle-ext}
\Ext^\bullet(F,G) \to \Ext^\bullet(r_1^*F, r_1^*G) \oplus \Ext^\bullet(r_2^*F, r_2^*G) \to \Ext^\bullet(r_0^*F, r_0^*G).
\end{equation} 
In particular, $F \in \Dperf(D)$ if and only if~$r_1^*F \in \Dperf(D_1)$ and~$r_2^*F \in \Dperf(D_2)$.
\item 
\label{item:functor-dperf}
If~$\cT$ is an enhanced triangulated category and $\varphi_1 \colon \cT \to \Dperf(D_1)$, $\varphi_2 \colon \cT \to \Dperf(D_2)$
is a pair of enhanced functors such that $i_1^* \circ \varphi_1 \cong i_2^* \circ \varphi_2$, 
there is a functor $\varphi \colon \cT \to \Dperf(D)$ such that~$r_1^* \circ \varphi \cong \varphi_1$ and~$r_2^* \circ \varphi \cong \varphi_2$.
\end{enumerate}
\end{lemma}

\begin{proof}
\eqref{item:triangle-object}
The triangle~\eqref{eq:triangle-object} is obtained by tensoring~$G$ with the canonical exact sequence
\begin{equation*}
0 \to \cO_D \xrightarrow{ (r_1^*,r_2^*)\ } \cO_{D_1} \oplus \cO_{D_2} \xrightarrow{\ (i_1^*,-i_2^*)\ } \cO_{D_0} \to 0.
\end{equation*}
Formation of~\eqref{eq:triangle-object} commutes with direct sums because the tensor product does.

\eqref{item:triangle-ext}
The triangle~\eqref{eq:triangle-ext} follows from~\eqref{eq:triangle-object} by an application of~$\Ext^\bullet(F,-)$ and adjunction.
If~$r_1^*F$ and~$r_2^*F$ are perfect, so is $r_0^*F \cong i_1^*r_1^*F \cong i_2^*r_2^*F$, 
therefore the second and third terms of~\eqref{eq:triangle-ext} commute with arbitrary direct sums in the second arguments.
Using~\eqref{eq:triangle-ext} we deduce that~$\Ext^\bullet(F,-)$ commutes with direct sums, therefore~$F$ is compact, and hence perfect.
The other implication for perfectness is obvious.

\eqref{item:functor-dperf}
This is~\cite[Chapter 8, Corollary~A.2.2]{GR17}.
\end{proof}

Fix a full triangulated subcategory $\cA \subset \Dperf(D_1)$ and denote by $i_1^*(\cA) \subset \Dperf(D_0)$ 
its image under the pullback functor~$i_1^* \colon \Dperf(D_1) \to \Dperf(D_0)$.
Consider the full subcategories
\begin{alignat}{3}
\label{def:cc}
\cC & \coloneqq \{ F &&\in \Dperf(D) &&\mid r_1^*F \in \cA \},\\
\label{def:cc2}
\cC_2 & \coloneqq \{ F_2 &&\in \Dperf(D_2) &&\mid i_2^*F_2 \in i_1^*(\cA) \}.
\end{alignat} 
We prove the following 

\begin{proposition}
\label{prop:cc-cc2-equivalence}
Assume the functor $i_1^* \colon \cA \to \Dperf(D_0)$ is fully faithful, 
so that $i_1^*(\cA) \subset \Dperf(D_0)$ is a full triangulated subcategory equivalent to~$\cA$.
Then the functor $r_2^* \colon \Dperf(D) \to \Dperf(D_2)$ induces an equivalence of categories $\cC \simeq \cC_2$ defined by~\eqref{def:cc} and~\eqref{def:cc2}.
\end{proposition}

The proof takes the rest of this subsection.
We start by defining a functor $\cC_2 \to \cC$.

\begin{lemma}
\label{lemma:varphi-definition}
Assume the functor $i_1^* \colon \cA \to \Dperf(D_0)$ is fully faithful.
There is a functor \mbox{$\varphi \colon \cC_2 \to \Dperf(D)$} such that $r_2^* \circ \varphi \colon \cC_2 \to \Dperf(D_2)$ is the natural embedding
and there is an isomorphism of functors
\begin{equation*}
i_1^* \circ r_1^* \circ \varphi \cong i_2^*\vert_{\cC_2} \colon \cC_2 \to \Dperf(D_0).
\end{equation*}
Moreover, $\varphi(\cC_2) \subset \cC$.
\end{lemma}

\begin{proof}
The category~$\cC_2$ is a full subcategory in~$\Dperf(D_2)$, so it has a natural enhancement;
therefore, to construct the required functor we can use Lemma~\ref{lemma:reducible}\eqref{item:functor-dperf}.
So, we need to construct a pair of functors~$\varphi_s \colon \cC_2 \to \Dperf(D_s)$
and an isomorphism of functors~$i_1^* \circ \varphi_1 \cong i_2^* \circ \varphi_2$, see the diagram
\begin{equation*}
\xymatrix@C=5em{
\cC_2 \ar@{^{(}->}[d]^-{\varphi_2} \ar[r]^(.4){i_2^*} \ar@{..>}@/^1ex/[rdd]^(.35){\varphi_1} \ar@{..>}@/_3em/[dd]_{\varphi} &
i_1^*(\cA) \ar@{^{(}->}[d] \ar@/^4em/[ddd]^-{(i_1^*)^{-1}} &
\\
\Dperf(D_2) \ar[r]_(.4){i_2^*} &
\Dperf(D_0) 
\\
\Dperf(D) \ar[u]_-{r_2^*} \ar[r]^(.4){r_1^*} &
\Dperf(D_1) \ar[u]_-{i_1^*} &
\\
\cC \ar@{_{(}->}[u] \ar[r] &
\cA. \ar@{_{(}->}[u] 
}
\end{equation*}
We take~$\varphi_2$ to be the natural embedding and define~$\varphi_1$ as the composition
\begin{equation*}
\varphi_1 \colon \cC_2 \xrightarrow{\ i_2^*\ } i_1^*(\cA) \xrightarrow{\ (i_1^*)^{-1}\ } \cA \hookrightarrow \Dperf(D_1)
\end{equation*}
(note that we use the assumption of full faithfulness of~$i_1^*$ on~$\cA$ to define the middle arrow).
Then we have a canonical isomorphism $i_1^* \circ \varphi_1 \cong i_2^* \circ \varphi_2$, 
hence there exists a functor~$\varphi \colon \cC_2 \to \Dperf(D)$ such that
\begin{equation*}
r_1^* \circ \varphi \cong \varphi_1
\qquad 
r_2^* \circ \varphi \cong \varphi_2.
\end{equation*}
Moreover, we have $i_1^* \circ r_1^* \circ \varphi \cong i_1^* \circ \varphi_1 \cong i_2^* \circ \varphi_2 \cong i_2^*\vert_{\cC_2}$, 
as required.

For the last claim note that $r_1^*(\varphi(\cC_2)) = \varphi_1(\cC_2) \subset \cA$ by construction;
this means that $\varphi(\cC_2) \subset \cC$.
\end{proof}

Next we check that the constructed functor is adjoint to~$r_2^*$.

\begin{lemma}
\label{lemma:varphi-adjunction}
Assume the functor $i_1^* \colon \cA \to \Dperf(D_0)$ is fully faithful.
The functor $\varphi \colon \cC_2 \to \cC$ constructed in Lemma~\textup{\ref{lemma:varphi-definition}} 
is right adjoint to the functor $r_2^* \colon  \cC \to \cC_2$.
\end{lemma}

\begin{proof}
First, note that $r_2^*(\cC) \subset \cC_2$; indeed, if $F \in \cC$ 
then~$r_1^*F \in \cA$ by~\eqref{def:cc}, hence
\begin{equation*}
i_2^*r_2^*F \cong i_1^*r_1^*F \in i_1^*(\cA),
\end{equation*}
and therefore~$r_2^*F \in \cC_2$ by~\eqref{def:cc2}.

Furthermore, let $F \in \cC$ and~$F_2 \in \cC_2$.
Using the construction of the functor~$\varphi$ from Lemma~\ref{lemma:varphi-definition} 
(that relies on Lemma~\ref{lemma:reducible}\eqref{item:functor-dperf})
and the triangle~\eqref{eq:triangle-ext} we obtain a distinguished triangle
\begin{equation*}
\Ext_D^\bullet(F,\varphi(F_2)) \xrightarrow{ (r_1^*,r_2^*)\ }
\Ext_{D_1}^\bullet(r_1^*F,\varphi_1(F_2)) \oplus 
\Ext_{D_2}^\bullet(r_2^*F,\varphi_2(F_2)) \xrightarrow{\ (i_1^*,-i_2^*)\ }
\Ext_{D_0}^\bullet(i_1^*r_1^*F,i_1^*(\varphi_1(F_2))), 
\end{equation*}
where we provide~$\Ext^\bullet$ with subscripts to emphasize which scheme we are working on. 
Recall also that~$\varphi_s \cong r_s^* \circ \varphi$.
Since $r_1^*F \in \cA$ by definition of~$\cC$ and $\varphi_1(F_2) \in \cA$ by construction of~$\varphi_1$, 
full faithfulness of~$i_1^*$ on~$\cA$ means that the morphism
\begin{equation*}
i_1^* \colon 
\Ext_{D_1}^\bullet(r_1^*F,\varphi_1(F_2)) \to
\Ext_{D_0}^\bullet(i_1^*r_1^*F,i_1^*(\varphi_1(F_2))) 
\end{equation*}
is an isomorphism.
Therefore,
\begin{equation*}
\Ext_{D}^\bullet(F,\varphi(F_2)) \cong 
\Ext_{D_2}^\bullet(r_2^*F,\varphi_2(F_2)),
\end{equation*}
proves the adjunction.
\end{proof}

Now we can deduce the proposition.

\begin{proof}[Proof of Proposition~\textup{\ref{prop:cc-cc2-equivalence}}]
Let~$\varphi \colon \cC_2 \to \cC$ be the functor constructed in Lemma~\ref{lemma:varphi-definition}.
By Lemma~\ref{lemma:varphi-adjunction} it is right adjoint to the functor~$r_2^* \colon \cC \to \cC_2$.
Moreover, we know from Lemma~\ref{lemma:varphi-definition} that~$r_2^* \circ \varphi$ 
is the natural embedding~$\varphi_2 \colon \cC_2 \hookrightarrow \Dperf(D_2)$.
This means that~$\varphi$ is fully faithful and it remains to check that~$\Ker(r_2^*\vert_\cC) = 0$.

So, assume~$F \in \cC$ is such that~$r_2^*F = 0$.
By~\eqref{diagram:reducible} we have
\begin{equation*}
i_1^*r_1^*F \cong i_2^*r_2^*F = 0.
\end{equation*}
Since~$r_1^*F \in \cA$ by definition of~$\cC$ and~$i_1^*$ is fully faithful on~$\cA$, it follows that~$r_1^*F = 0$.
Finally, the equalities~$r_1^*F = 0$, $r_2^*F = 0$, and~$r_0^*F \cong i_2^*r_2^*F = 0$
imply that~$F = 0$ by Lemma~\ref{lemma:reducible}\eqref{item:triangle-object}.
\end{proof}

\subsection{The main result}
\label{subsection:proof}

In this subsection we prove that the base change~$\cD_o$ of the category~$\cD$ defined by~\eqref{eq:def-cd} is smooth and proper over~$\kk$
if the Lefschetz decomposition~\eqref{eq:Lefschetz-e} satisfies appropriate conditions 
and deduce that in this case~$\cD$ is smooth and proper over~$B$.

We start with a simple observation that allows us to apply results from~\S\ref{subsection:reducible}.

\begin{lemma}
\label{lemma:ies}
Assume a left Lefschetz decomposition~\eqref{eq:Lefschetz-e} is given.
The functor $i_E^* \colon \Db(E) \to \Db(E_o)$ is fully faithful on the subcategories~$\cA_k \subset \Db(E)$ for~$1 -m \le k \le -1$
and preserves their semiorthogonality, so that we have a subcategory with a semiorthogonal decomposition
\begin{equation*}
\langle \cA'_{1-m} \otimes L^{2-m}, \dots, \cA'_{-2} \otimes L^{-1}, \cA'_{-1} \rangle \subset \Db(E_o),
\end{equation*}
where we denote $\cA'_{k} \coloneqq i_E^*(\cA_k) \subset \Db(E_o)$ and abbreviate $L\vert_{E_o}$ to just~$L$.
\end{lemma}
\begin{proof}
By Lemma~\ref{lemma:e0-e} the divisor $E_o$ corresponds to a section of the line bundle~$L$, 
so both full faithfulness and semiorthogonality follow from the standard properties of Lefschetz decompositions,
see~\cite[Theorem~6.3]{K07} or~\cite[Proposition~2.4]{K08c}.
\end{proof}

Now we can state the main result of the paper.
Recall the notation introduced in~\eqref{eq:notationn-b-o}, \eqref{def:l}, and~\eqref{diagram:divisor}, 
especially the line bundle~$L$ and the maps~$i_X \colon E_o \to X'_o$ and~$i_E \colon E_o \to E$.

\begin{theorem}
\label{theorem:simultaneous}
Let~$f \colon X \to B$ be a flat projective morphism to a smooth pointed curve~$(B,o)$ such that~$X^o$ is smooth over~$B^o$ 
and let~$Z \subset X_o$ be a smooth closed subscheme.
Assume~$X$ has rational singularities, the blowups~$\tX = \Bl_{Z}(X)$, $X'_o = \Bl_{Z}(X_o)$,
and their exceptional divisors~$E$ and~$E_o$ are all smooth, and the central fiber~\eqref{eq:tx-o} of~$\tX \to B$ is reduced.
Let~\eqref{eq:Lefschetz-e} be a~$Z$-linear left Lefschetz decomposition of~$\Db(E)$ such that~\eqref{eq:co-e-assumption} holds.
Let~$\cD$ be the subcategory of~$\Db(\tX)$
defined by~\eqref{eq:def-cd}.
Assume additionally that the components~$\cA_k$ of~\eqref{eq:Lefschetz-e} satisfy the condition
\begin{equation}
\label{eq:ca1}
\cA_{-1} = \cA_0
\end{equation}
and that the embedding of Lemma~\textup{\ref{lemma:ies}} is an equality, i.e., 
\begin{equation}
\label{eq:Lefschetz-e0}
\Db(E_o) =
\langle \cA'_{1-m} \otimes L^{2-m}, \dots, \cA'_{-2} \otimes L^{-1}, \cA'_{-1} \rangle.
\end{equation}
Then 
\begin{enumerate}\renewcommand{\theenumi}{\roman{enumi}}
\item 
\label{theorem:item:cd-o}
The base change~$\cD_o$ of~$\cD$ along the embedding~$\{o\} \hookrightarrow B$ is smooth and proper over~$\kk$, and
\begin{equation}
\label{eq:cd-o}
\cD_o \simeq \{F \in \Db(X'_o) \mid i_X^*F \in \cA'_{-1} \} \subset \Db(X'_o).
\end{equation} 
Moreover, there is a semiorthogonal decomposition
\begin{equation}
\label{eq:sod-db-xprime-o}
\Db(X'_o) = \langle i_{X*}(\cA'_{1-m} \otimes L^{2-m}), \dots, i_{X*}(\cA'_{-2} \otimes L^{-1}), \cD_o \rangle.
\end{equation}
\item 
\label{theorem:item:cd}
The triple $(\cD,\pi^*,\pi_*)$ is a simultaneous categorical resolution of~$(X,X_o)$; 
in particular~$\cD$ is smooth and proper over~$B$.
\end{enumerate}
\end{theorem}

\begin{proof}
Note that the components~$\cA_k$ of~\eqref{eq:Lefschetz-e} and~$\cA'_k$ of~\eqref{eq:Lefschetz-e0} are admissible 
because~$E$ and~$E_o$ are smooth, 
so the assumptions of Theorem~\ref{theorem:categorical} are satisfied,
both for the blowup~$\pi \colon \tX = \Bl_Z(X) \to X$ with Lefschetz decomposition~\eqref{eq:Lefschetz-e}
and for the blowup~$\pi' \colon X'_o = \Bl_Z(X_o) \to X_o$ with Lefschetz decomposition~\eqref{eq:Lefschetz-e0}.
The first implies that~$\cD$ is a categorical resolution for~$X$ and Proposition~\ref{prop:cc-perf-b} holds.
The latter implies that the category
\begin{equation*}
\cD' \coloneqq \{F \in \Db(X'_o) \mid i_X^*F \in \cA'_{-1} \} \subset \Db(X'_o) 
\end{equation*}
provides a categorical resolution for~$X'_o$; in particular that~$\cD'$ is smooth and proper over~$\kk$,
and that there is a semiorthogonal decomposition~\eqref{eq:sod-db-xprime-o} with~$\cD_o$ replaced by~$\cD'$.
So, to prove part~\eqref{theorem:item:cd-o} of the theorem we only need to check that~$\cD_o \simeq \cD'$.

We start by checking that $(\cD_o)^\pf \simeq \cD'$.
Recall that by Proposition~\ref{prop:cc-perf-b} the category~$(\cD_o)^\pf$ is described by~\eqref{eq:ccb-characterization},
so we need to identify the right side of~\eqref{eq:ccb-characterization} with~$\cD'$.
For this we apply the results of~\S\ref{subsection:reducible}.
We consider the reducible scheme $D = E \cup X'_o$ with $D_1 = E$, $D_2 = X'_o$, so that~$D_0 = E_o$.
Furthermore, we consider the subcategory~$\cA \coloneqq \cA_0 = \cA_{-1} \subset \Db(E) = \Dperf(E)$. 
Note that the functor~$i_E^* \colon \cA \to \Dperf(E_o)$ is fully faithful by~\eqref{eq:ca1} and Lemma~\ref{lemma:ies} and its image is~$\cA'_{-1}$. 
Thus, Proposition~\ref{prop:cc-cc2-equivalence} applies and proves 
\begin{equation*}
(\cD_o)^\pf \simeq \cD';
\end{equation*}
indeed, the above definition of the category~$\cD'$ agrees with~\eqref{def:cc2} because~$\Db(X'_o) = \Dperf(X'_o)$ as~$X'_o$ is smooth.
In particular, it follows that the category~$(\cD_o)^\pf$ is smooth and proper over~$\kk$ because~$\cD'$ is.
Now we apply Proposition~\ref{prop:perf-coh} which implies that~$\cD_o = (\cD_o)^\pf$ and hence $\cD_o \simeq \cD'$.
This proves~\eqref{theorem:item:cd-o}.

To prove~\eqref{theorem:item:cd} the only thing we need to show is that~$\cD$ is smooth and proper over~$B$ 
(all the rest follows from Theorem~\ref{theorem:categorical}).
For this we use Theorem~\ref{theorem:sp}.
So, let $B_{\mathrm{sm,pr}} \subset B$ be the locus of smooth and proper fibers of~$\cD$ over~$B$.
For~$b \in B^o$ the fiber~$\cD_b$ is equivalent to~$\Db(X_b)$ by Lemma~\ref{lemma:cd-b}, hence smooth and proper by Lemma~\ref{lemma:sp-geometric}.
On the other hand, the category~$\cD_o$ is smooth and proper over~$\kk$ by part~\eqref{theorem:item:cd-o}.
Therefore, the open subscheme~$B_{\mathrm{sm,pr}} \subset B$ contains all points of~$B$, hence $B_{\mathrm{sm,pr}} = B$,
and so~$\cD$ is smooth and proper over~$B$ by Theorem~\ref{theorem:sp}.
\end{proof}

\begin{remark}
\label{remark:divisorian-degeneration}
The same result can be proved in a more general situation when the base scheme~$B$ has arbitrary dimension,
the flat projective morphism~$f \colon X \to B$ is smooth over the complement $B^o = B \setminus B_o$ of a Cartier divisor~$B_o \subset B$,
the subscheme~$Z \subset X_o \coloneqq X \times_B B_o$ is such that all fibers of~$f$ are smooth away from~$Z$,
the blowup~$\tX = \Bl_Z(X)$ is smooth over~$\kk$, 
the blowup~$\Bl_Z(X_o)$ is smooth over~$B_o$,
and the other assumptions are the same as in Theorem~\ref{theorem:simultaneous}.
\end{remark}

To find examples where Theorem~\ref{theorem:simultaneous} applies 
and produces simultaneous categorical resolutions we need fibrations~$E \to Z$
with relatively ample line bundle~$L$ and~$Z$-linear Lefschetz decomposition
satisfying the properties~\eqref{eq:co-e-assumption}, \eqref{eq:ca1}, and~\eqref{eq:Lefschetz-e0}.
Below we list some examples of such varieties over the trivial base (over more general bases one should consider their relative analogues) 
and their Lefschetz decompositions.

\begin{example}
\label{ex:scr}
 The following Lefschetz decompositions enjoy the properties~\eqref{eq:co-e-assumption}, \eqref{eq:ca1} and~\eqref{eq:Lefschetz-e0} 
 (for any \emph{smooth} divisor~$E_o \subset E$ in the linear system~$|L|$):
\begin{enumerate}
\item
\label{example:pn}
$E = \P^n$, $L = \cO(1)$, $\cA_0 = \cA_{-1} = \dots = \cA_{-n} = \langle \cO \rangle$;
\item
\label{example:segre}
$E = \P^{n_1} \times \P^{n_2}$, $L = \cO(1,1)$, $\cA_0 = \cA_{-1} = \dots = \cA_{-n_2} = p_1^*(\Db(\P^{n_1}))$, 
where~$p_1$ is the projection~$\P^{n_1} \times \P^{n_2} \to \P^{n_1}$ and we assume~$n_1 \le n_2$
(if~$n_1 > n_2$ a similar Lefschetz decomposition exists, but it does not have the property~\eqref{eq:Lefschetz-e0});
\item
\label{example:q2n}
$E = Q^{2n}$, $L = \cO(1)$, $\cA_0 = \cA_{-1} = \langle \cS, \cO \rangle$, $\cA_{-2} = \dots = \cA_{1-2n} = \langle \cO \rangle$, 
where $\cS$ is one of the two spinor bundles on the smooth even-dimensional quadric~$Q^{2n}$
(if~$E$ is a quadric of odd dimension a similar Lefschetz decomposition exists, see Example~\ref{ex:nodal},
but it does not have the property~\eqref{eq:ca1});
\item
\label{example:gr2n}
$E = \Gr(2,n)$, $L = \cO(1)$, 
$\cA_0 = \cA_{-1} = \dots = \cA_{-n} = \langle \cO, \cU^\vee, \dots, S^{\lfloor n/2 \rfloor - 1}\cU^\vee \rangle$ if~$n$ is odd,
and with a slightly more complicated formula if~$n$ is even, see~~\cite[Theorem~4.1]{K08c},
where $\cU$ is the tautological bundle on the Grassmannian~$\Gr(2,n)$.
\end{enumerate}

\end{example}

In case~\eqref{example:pn} the corresponding varieties $X$ and $X_o$ are smooth and it is easy to see 
that the categorical resolution of Theorem~\ref{theorem:simultaneous} is trivial (i.e., $\cD \simeq \Db(X)$, $\cD_o \simeq \Db(X_o)$).

Case~\eqref{example:segre} appears in the standard flipping contraction.
In this case the category~$\cD$ is equivalent to the derived category of the flip.

Case~\eqref{example:q2n} is the most interesting.
It corresponds to the assumption that~$X$ and~$X_o$ have ordinary double points and~$\dim(X_o)$ is~\emph{even}.
This is precisely the situation of Theorem~\ref{theorem:simultaneous-nodal} from the Introduction.
We will prove a more general result in the next subsection.

\subsection{Simultaneous categorical resolutions of nodal degenerations}
\label{subsection:examples}

Recall from~\cite[\S3.5]{K08b} that if~$p \colon E \to Z$ is a fibration in smooth quadrics of even dimension~$2n$,
then there is an \'etale double covering~$\tZ \to Z$, a Brauer class~$\beta \in \Br(\tZ)$ (represented by an explicit Azumaya algebra),
and a $Z$-linear Lefschetz decomposition
\begin{equation}
\label{eq:db-quadric-even}
\Db(E) = \langle p^*\Db(Z) \otimes \cO_E(1 - 2n), \dots, p^*\Db(Z) \otimes \cO_E(-1), 
p^*\Db(\tZ,\beta) \otimes \cS_E, p^*\Db(Z) \otimes \cO_E \rangle,
\end{equation}
where~$\cS_E$ is a~$\beta^{-1}$-twisted spinor bundle on~$E \times_Z \tZ$. 
We call~$\tZ \to Z$ and~$\beta$ the {\sf discriminant double cover} and {\sf Brauer class} of~$E/Z$.

Similarly, if~$p \colon E' \to Z$  is a fibration in smooth quadrics of odd dimension~$2n - 1$ then by~\cite[\S3.6]{K08b}
there is a Brauer class~$\beta \in \Br(Z)$ 
and a $Z$-linear Lefschetz decomposition
\begin{equation}
\label{eq:db-quadric-odd}
\Db(E') = \langle p^*\Db(Z) \otimes \cO_{E'}(2 - 2n), \dots, p^*\Db(Z) \otimes \cO_{E'}(-1), 
p^*\Db(Z,\beta) \otimes \cS_{E'}, p^*\Db(Z) \otimes \cO_{E'} \rangle,
\end{equation}
where~$\cS_{E'}$ is a~$\beta^{-1}$-twisted spinor bundle on~$E'$.

We say that a scheme~$Y$ is {\sf nodal} along a smooth subscheme~$Z \subset Y$, 
if~$\Sing(Y) = Z$ and a transverse slice to~$Z$ in~$Y$ at any point of~$Z$ has an ordinary double point.

\begin{theorem}
\label{theorem:nodal-general}
Let~$\kk$ be a field of characteristic not equal to~$2$.
Let $f \colon X \to B$ be a flat projective morphism to a smooth pointed $\kk$-curve $(B,o)$ such that~$X^o$ is smooth over~$B^o$.
Let~$Z \subset X_o$ be a smooth closed subscheme in the central fiber of~$f$ 
such that both~$X$ and~$X_o$ are nodal along~$Z$.
Assume
\begin{equation*}
\codim_Z(X_o) = 2n
\end{equation*}
is even.
Let~$\eps \colon E \to \tX = \Bl_Z(X)$ be the exceptional divisor of the blowup.
Let~$p \colon E \to Z$ be the corresponding fibration in smooth quadrics of dimension~$2n$.
Let~$\tZ \to Z$ be the corresponding discriminant \'etale double covering 
and let~$\beta \in \Br(\tZ)$ be the corresponding Brauer class.

If the covering~$\tZ \to Z$ splits then there is a semiorthogonal decomposition
\begin{equation}
\label{eq:sod-nodal-general}
\Db(\tX) = \langle \eps_*(p^*\Db(Z) \otimes \cO_E((2n-1)E)), \dots, \eps_*(p^*\Db(Z) \otimes \cO_E(E)), 
\eps_*(p^*\Db(Z,\beta) \otimes \cS), \cD \rangle ,
\end{equation}
where~$\cS$ is a~$\beta^{-1}$-twisted spinor bundle on~$E$
and the category~$\cD$ defined by this decomposition  provides a simultaneous categorical resolution of~$(X,X_o)$.
Moreover, the central fiber~$\cD_o$ of the category~$\cD$ fits into a semiorthogonal decomposition
\begin{equation*}
\Db(\Bl_Z(X_o)) = \langle i_{X*}(p_o^*\Db(Z) \otimes \cO_{E_o}((2n-2)E_o)), \dots, i_{X*}(p_o^*\Db(Z) \otimes \cO_{E_o}(E_o)), \cD_o \rangle,
\end{equation*}
where $i_X \colon E_o \hookrightarrow \Bl_Z(X_o)$ is the exceptional divisor of the blowup~$\Bl_Z(X_o)$ 
and~$p_o \colon E_o \to Z$ is the natural projection.
\end{theorem}

    \begin{proof}
Set $\cO_E(1) \coloneqq \cO_E(-E)$
and consider the semiorthogonal decomposition~\eqref{eq:db-quadric-even}.
If the covering~$\tZ \to Z$ splits as~$\tZ = Z_1 \sqcup Z_2$, 
the ``spinor'' component of the semiorthogonal decomposition also splits as
\begin{equation*}
p^*\Db(\tZ,\beta) \otimes \cS_{E} = \langle p^*\Db(Z_1,\beta) \otimes \cS_1, p^*\Db(Z_2,\beta) \otimes \cS_2 \rangle,
\end{equation*}
where~$\cS_i$ are the restrictions of~$\cS_{E}$ to~$E \times_Z Z_i \cong E$.
Mutating~$p^*\Db(Z_2,\beta) \otimes \cS_2$ to the right of~$p^*\Db(Z) \otimes \cO_E$
we obtain~$p^*\Db(Z_1,\beta) \otimes \cS_1(1)$ 
(this follows from a relative version~\cite[Lemma~4.5]{K08b} of the exact sequences~\eqref{eq:spinor-sequences}).
Thus, identifying the component~$Z_1$ of~$\tZ$ with~$Z$ and writing~$\cS$ instead of~$\cS_1$,
we can rewrite the obtained semiorthogonal decomposition of~$\Db(E)$ as
the Lefschetz decomposition~\eqref{eq:Lefschetz-e} with
\begin{equation}
\label{eq:Lefschetz-quadric}
\cA_0 = \cA_{-1} = \langle p^*\Db(Z) \otimes \cO_E, p^*\Db(Z,\beta) \otimes \cS(1) \rangle
\quad\text{and}\quad 
\cA_{-2} = \dots = \cA_{1-2n} = p^*\Db(Z) \otimes \cO_E.
\end{equation}
In particular, the condition~\eqref{eq:ca1} holds.
Moreover, the condition~\eqref{eq:Lefschetz-e0} follows from~\eqref{eq:db-quadric-odd}.
Now Theorem~\ref{theorem:simultaneous} applies, providing the required semiorthogonal decompositions 
of the categories~$\Db(\tX)$ and~$\Db(\Bl_Z(X_o))$
and proving that the category~$\cD = \{ F \in \Db(\tX) \mid \eps^*F \in \cA_0 \}$ fits into~\eqref{eq:sod-nodal-general} 
and provides a simultaneous categorical resolution of~$(X,X_o)$.
\end{proof}

Now we can deduce Theorem~\ref{theorem:simultaneous-nodal}.

\begin{proof}[Proof of Theorem~\textup{\ref{theorem:simultaneous-nodal}}]
We apply Theorem~\ref{theorem:nodal-general} with~$Z = \{x_o\} \cong \Spec(\kk)$.
If the double covering~$\tZ \to Z$ splits (this is certainly true if~$\kk$ is algebraically closed)  
we obtain the required simultaneous categorical resolution.
Otherwise, $\tZ = \Spec(\kk')$, where~$\kk'/\kk$ is a quadratic extension,
hence after extension of scalars to~$\kk'$ the double covering splits, and the same argument applies.
\end{proof}

In the setup of Theorem~\ref{theorem:simultaneous-nodal}, if the field~$\kk$ is algebraically closed, 
hence the double covering splits and the Brauer class vanishes, 
the semiorthogonal decomposition~\eqref{eq:sod-nodal-general} takes the form
\begin{equation}
\label{eq:sod-nodal}
\Db(\tX) = \langle \cO_E(1-2n), \dots, \cO_E(-2), \cO_E(-1), \eps_*\cS, \cD \rangle,
\end{equation}
where recall that~$\eps \colon E \hookrightarrow \tX = \Bl_{x_o}(X)$ is the embedding of the exceptional divisor 
and~$\cS$ is a spinor bundle on the smooth quadric~$E$.
Note that the subcategory~$\cD \subset \Db(\tX)$ defined by~\eqref{eq:sod-nodal} depends on the choice of the spinor bundle~$\cS$.
In the rest of this subsection we show that the two different choices 
result in equivalent categories related by a ``categorical flop'' (see Remark~\ref{rem:cat-flop}).

Recall that the two spinor bundles~$\cS_+$ and~$\cS_-$ on a quadric of dimension~$2n$ are completely orthogonal 
and are related by the exact sequences
\begin{equation}
\label{eq:spinor-sequences}
0 \to \cS_- \to \cO^{\oplus N} \to \cS_+(1) \to 0
\qquad\text{and}\qquad 
0 \to \cS_+ \to \cO^{\oplus N} \to \cS_-(1) \to 0,
\end{equation}
where $N = 2^n$.

\begin{proposition}
Let~$\kk$ be an algebraically closed field of characteristic not equal to~$2$.
Assume~$X$ has an ordinary double point at~$x_0$ and is smooth elsewhere.
Let~$\tX = \Bl_{x_o}(X)$ be the blowup with the exceptional divisor~$\eps \colon E \hookrightarrow \tX$ 
\textup(isomorphic to a smooth quadric of dimension~$2n$\textup).
Let~$\cS_+$ and~$\cS_-$ be the two spinor bundles on~$E$ and let~$\cD_+,\cD_- \subset \Db(\tX)$ 
be the categorical resolutions defined by~\eqref{eq:sod-nodal}
for the choices~$\cS = \cS_+$ and~$\cS = \cS_-$, respectively.
Then the following are true:
\begin{enumerate}\renewcommand{\theenumi}{\roman{enumi}}
\item 
\label{item:es-cspm}
The objects~$\eps_*\cS_+$ and~$\eps_*\cS_-$ in~$\Db(\tX)$ are exceptional, and, moreover,
\begin{equation*}
\eps_*\cS_- \in 
\cD_-^\perp \cap {}^\perp\cD_+,
\qquad 
\eps_*\cS_+ \in 
\cD_+^\perp \cap {}^\perp\cD_-.
\end{equation*}
\item 
\label{item:cepm}
The right mutations~$\cK_+ \coloneqq \bR_{\eps_*\cS_+}(\eps_*\cS_-)$ and~$\cK_- \coloneqq \bR_{\eps_*\cS_-}(\eps_*\cS_+)$ 
fit into distinguished triangles
\begin{equation}
\label{eq:ce_pm}
\cK_+ \to \eps_*\cS_- \to \eps_*\cS_+[2],
\qquad\text{and}\qquad 
\cK_- \to \eps_*\cS_+ \to \eps_*\cS_-[2],
\end{equation}
and we have~$\cK_+ \in \cD_+$ and~$\cK_- \in \cD_-$.
\item 
\label{item:equivalences}
The right mutation functors
\begin{equation*}
\Phi_- \coloneqq \bR_{\eps_*\cS_-} \colon \cD_+ \to \cD_-
\qquad\text{and}\qquad 
\Phi_+ \coloneqq \bR_{\eps_*\cS_+} \colon \cD_- \to \cD_+
\end{equation*}
are equivalences of categories.
Moreover, the objects~$\cK_\pm$ are spherical and the compositions of equivalences~$\Phi_+ \circ \Phi_-$ and~$\Phi_- \circ \Phi_+$
are isomorphic to the spherical twists with respect to~$\cK_+$ and~$\cK_-$, respectively.
\end{enumerate}
\end{proposition}

\begin{proof}
We will denote by~$\cA_k^\pm$, $1 - 2n \le k \le 0$, the components of the semiorthogonal decompositions~\eqref{eq:Lefschetz-quadric}
(cf.~Example~\ref{ex:nodal})
corresponding to the choices~$\cS = \cS_+$ and~$\cS = \cS_-$, respectively, so that in particular
\begin{equation}
\label{eq:ca0-pm}
\cA_0^+ = \cA_{-1}^+ \coloneqq \langle \cO, \cS_+(1) \rangle = \langle \cS_-, \cO \rangle
\qquad\text{and}\qquad 
\cA_0^- = \cA_{-1}^- \coloneqq \langle \cO, \cS_-(1) \rangle = \langle \cS_+, \cO \rangle,
\end{equation}
where the equalities in the right hand sides follow from~\eqref{eq:spinor-sequences}.

First, recall that~$\cS_+$ and~$\cS_-$ are exceptional and completely orthogonal on~$E$.
Moreover, $\cS_\mp \in \cA^\pm_{-1}$, hence the functor~$\eps_*$ is fully faithful on them by Theorem~\ref{theorem:categorical};
thus the sheaves~$\eps_*\cS_+$ and~$\eps_*\cS_-$ are exceptional.

Next, we note that $\eps_*\cS_\pm \in \cD_\pm^\perp$ by~\eqref{eq:sod-nodal} for~$\cS = \cS_\pm$.
Let us prove that~$\eps_*\cS_- \in {}^\perp\cD_+$.
For this take~\mbox{$F \in \cD_+$} and note that since the normal bundle of~$E$ is~$\cO_E(-1)$
it follows from Grothendieck duality that
\begin{equation*}
\Ext^\bullet(\eps_*\cS_-, F) \cong
\Ext^\bullet(\cS_-, \eps^!F) \cong
\Ext^\bullet(\cS_-, \eps^*F(-1)[-1]).
\end{equation*}
Since~$\cS_- \in \cA_0^+$ and~$\eps^*F(-1) \in \cA_0^+(-1) = \cA_{-1}^+(-1)$ by~\eqref{eq:ca0-pm},
the above space vanishes, hence indeed we have~\mbox{$\eps_*\cS_- \in {}^\perp\cD_+$}.
The containment~$\eps_*\cS_+ \in {}^\perp\cD_-$ is analogous, so~\eqref{item:es-cspm} is proved.

Next, 
since the conormal bundle of~$E$ is~$\cO_E(1)$, we have the distinguished triangle
\begin{equation*}
\eps^*\eps_*\cS_\pm \to \cS_\pm \to \cS_\pm(1)[2].
\end{equation*}
Using this triangle combined with adjunction and orthogonality of~$\cS_+$ and~$\cS_-$ on~$E$, we obtain
\begin{equation*}
\Ext^\bullet(\eps_*\cS_\pm,\eps_*\cS_\mp) \cong 
\Ext^\bullet(\cS_\pm(1)[1],\cS_\mp).
\end{equation*}
From~\eqref{eq:spinor-sequences} we conclude that the right side is equal to~$\kk[-2]$, hence
\begin{equation}
\label{eq:ext-cs-pm}
\Ext^\bullet(\eps_*\cS_\pm,\eps_*\cS_\mp) \cong \kk[-2].
\end{equation} 
This shows that the mutation triangles defining~$\cK_+ \coloneqq \bR_{\eps_*\cS_+}(\eps_*\cS_-)$ 
and~$\cK_- \coloneqq \bR_{\eps_*\cS_-}(\eps_*\cS_+)$ 
take the form~\eqref{eq:ce_pm}.

It follows from~\eqref{eq:sod-nodal} for~$\cS = \cS_+$ that to prove the containment~$\cK_+ \in \cD_+$ 
we must check that~$\cK_+$ is left-orthogonal to~$\cO_E(-i)$ with $1 \le i \le 2n-1$ and to~$\eps_*\cS_+$.
The first part of orthogonality is immediate from~\eqref{eq:ce_pm} and the two decompositions~\eqref{eq:sod-nodal} for~$\cS = \cS_+$ and~$\cS_-$,
while the second part follows from the definition of~$\cK_+$ as right mutation.
The containment $\cK_- \in \cD_-$ is proved analogously.
This proves~\eqref{item:cepm}.

To prove~~\eqref{item:equivalences} set
\begin{equation*}
\tcD = {}^\perp \langle \cO_E(1-2n), \dots, \cO_E(-2), \cO_E(-1) \rangle \subset \Db(\tX).
\end{equation*}
The semiorthogonal decompositions~\eqref{eq:sod-nodal} for~$\cS = \cS_+$ and~$\cS = \cS_-$ imply that
\begin{equation*}
\tcD = \langle \eps_*\cS_+, \cD_+ \rangle = \langle \eps_*\cS_-, \cD_- \rangle.
\end{equation*}
On the other hand, we also have 
\begin{equation*}
\tcD = \langle \cD_-, \eps_*\cS_+ \rangle = \langle \cD_+, \eps_*\cS_- \rangle;
\end{equation*}
indeed, semiorthogonality of the components is proved in part~\eqref{item:es-cspm} and generation follows from~\eqref{eq:ce_pm} 
and the containments~$\cK_\pm \in \cD_\pm$ proved in part~\eqref{item:cepm}.
So, we see that the functors~$\Phi_\pm$ are just the equivalences
\begin{equation*}
(\eps_*\cS_+)^\perp \xrightarrow{\ \sim\ } {}^\perp(\eps_*\cS_+)
\qquad\text{and}\qquad
(\eps_*\cS_-)^\perp \xrightarrow{\ \sim\ } {}^\perp(\eps_*\cS_-)
\end{equation*}
given by the right mutation functors in~$\tcD$.
The fact that their compositions are given by spherical twists follows from~\cite[Theorem~3.11]{HLS}
since the above decompositions of~$\tcD$ form a 4-periodic sequence of semiorthogonal decompositions.
\end{proof}

\begin{remark}
\label{rem:cat-flop}
Note that when~\mbox{$n = 1$} the simultaneous categorical resolution of a nodal surface degeneration 
constructed in Theorem~\ref{theorem:simultaneous-nodal} 
reduces to the geometric simultaneous resolution of Example~\ref{example:geometric}.
Indeed, the blowup~$\tX$ of~$X$ at~$x_o$ is isomorphic to the blowup of a small resolution of singularities~$X$ at its exceptional curve;
and comparing the blowup formula with the semiorthogonal decomposition~\eqref{eq:sod-nodal} (note that~$\cS = \cO_E(-1,0)$ in this case) 
gives an identification of the simultaneous categorical resolution~$\cD$ 
with the derived category of the geometric simultaneous resolution.
Furthermore, the categorical flop described above is equivalent to the standard Atiyah flop.
\end{remark}

\section{Application to cubic fourfolds}
\label{sec:cubic4}

In this section~$\kk$ is an algebraically closed field of characteristic not equal to~$2$.

For any cubic hypersurface $Y \subset \P^5$ there is a semiorthogonal decomposition~\cite{K10:cubic}
\begin{equation}
\label{eq:sod-dby}
\Db(Y) = \langle \cA_Y, \cO_Y, \cO_Y(1), \cO_Y(2) \rangle.
\end{equation}
If $Y$ is smooth the component~$\cA_Y$ is smooth and proper and it is a {\sf $K3$ category}
(i.e., its Serre functor is isomorphic to the shift functor~$[2]$), 
see~\cite[Corollary~4.4]{K04:cubic} or~\cite[Corollary~4.1]{K19}.

On the other hand, if $Y$ has a single ordinary double point~$y_o$, the category~$\cA_Y$ is no longer smooth and proper, 
but it has a nice categorical resolution provided by the derived category of a $K3$ surface~$S$
(it is also, of course, a $K3$ category), which is naturally associated to~$Y$, see~\cite[Theorem~5.2]{K10:cubic}.

More precisely, the linear projection $Y \dasharrow \P^4$ out of~$y_o$ induces an isomorphism
\begin{equation*}
\Bl_{y_o}(Y) \cong \Bl_S(\P^4),
\end{equation*}
where $S$ is a smooth complete intersection of a smooth 3-dimensional quadric~$Q \subset \P^4$ and a cubic hypersurface in~$\P^4$,
and the exceptional divisor of~$\Bl_{y_o}(Y)$ maps isomorphically onto~$Q$.
In this situation Theorem~\ref{theorem:categorical} shows that 
on the one hand, the categorical resolution~$\cD_Y$ of~$Y$ constructed in Example~\ref{ex:nodal} 
is a component of the semiorthogonal decomposition
\begin{equation}
\label{eq:db-bl-yo-y}
\Db(\Bl_{y_o}(Y)) = \langle \cO_Q(-2), \cO_Q(-1), \cD_Y \rangle,
\end{equation}
where~$Q$ is identified with the exceptional divisor of the blowup,
and on the other hand, by~\cite[Theorem~5.2]{K10:cubic} it has a semiorthogonal decomposition
\begin{equation}
\label{eq:cd-y}
\cD_Y = \langle \tilde\cA_Y, \cO_{\Bl_{y_o}(Y)}, \cO_{\Bl_{y_o}(Y)}(1), \cO_{\Bl_{y_o}(Y)}(2) \rangle,
\end{equation}
where the line bundles are pulled back from~$Y$ and
\begin{equation}
\label{eq:tca-y}
\tilde\cA_Y \simeq \Db(S).
\end{equation} 
In this section we prove Corollary~\ref{cor:cubic} by showing that $K3$ categories of both types 
($\cA_Y$ for smooth cubic fourfolds and~$\tilde\cA_Y$ for nodal ones)
fit into a single smooth and proper family.

First, we construct an appropriate geometric family of cubic fourfolds.

\begin{lemma}
\label{lemma:cubic-family}
For any cubic fourfold~$Y$ with a single ordinary double point~$y_o$ 
there is a family $f \colon X \to B$ of cubic fourfolds over a smooth pointed curve~$(B,o)$
in which the central fiber is~$X_o \cong Y$, the point~$y_o \in X$ is an ordinary double point, 
and the morphism~$f$ is smooth over~$B^o$.
\end{lemma}
\begin{proof}
Let $F_0(x)$ and $F_1(x)$ be the equations of $Y \subset \P^5$ 
and of a smooth cubic fourfold $Y_1 \subset \P^5$ (in the same projective space) 
such that~$Y_1$ does not contain the singular point~$y_o \in Y$.
Consider the family
\begin{equation*}
X = \{ (x,t) \in \P^5 \times \AA^1 \mid (1-t^2)F_0(x) + t^2F_1(x) = 0 \}
\end{equation*}
of cubic fourfolds over~$\AA^1$. 
The fiber of~$X$ over~$0 \in \AA^1$ is~$Y$ and the fiber over~$1 \in \AA^1$ is~$Y_1$.
Since~$Y_1$ is smooth, the general fiber of the projection~$X \to \AA^1$ is smooth as well, hence the set of points~$t \in \AA^1$ 
such that the fiber~$X_t$ is singular is a finite set $T \subset \AA^1$ containing the point~$0$.
Furthermore, it is easy to see that~$y_o \in X$ is an ordinary double point of the total space.
Now we let 
\begin{equation*}
B \coloneqq \AA^1 \setminus ( T \setminus \{0\} )
\end{equation*}
be the complement of all points in~$T$ except for~$0$.
Then the morphism~$X_B \to B$ has the prescribed properties.
\end{proof}

Now, applying Theorem~\ref{theorem:simultaneous-nodal} to the family $X \to B$ constructed in Lemma~\ref{lemma:cubic-family}, 
we prove Corollary~\ref{cor:cubic}.

\begin{proof}[Proof of Corollary~\textup{\ref{cor:cubic}}]
Let~$\tX = \Bl_{y_o}(X)$ and let~$\cD \subset \Db(\tX)$ be the simultaneous categorical resolution of~$(X,X_o = Y)$
constructed in Theorem~\ref{theorem:simultaneous-nodal} (that relies on Theorem~\ref{theorem:nodal-general}), so that
\begin{equation*}
\Db(\tX) = \langle \cO_E(-3), \cO_E(-2), \cO_E(-1), \eps_*\cS, \cD \rangle.
\end{equation*}
Then we have
\begin{equation*}
\cD_{b} \simeq \Db(X_{b})
\qquad\text{for $b \ne o$,}
\end{equation*}
and~$\cD_o$ is the categorical resolution of the singular cubic fourfold~$Y$ that fits into the semiorthogonal decomposition
\begin{equation*}
\Db(\Bl_{y_o}(Y)) = \langle \cO_{E_o}(-2), \cO_{E_o}(-1), \cD_o \rangle,
\end{equation*}
where $E_o$ is the exceptional divisor of the blowup~$\Bl_{y_o}(Y)$, i.e., $E_o = Q$.
Therefore, $\cD_o$ coincides with the categorical resolution~$\cD_Y$ of~$Y$ defined by~\eqref{eq:db-bl-yo-y}.
Note further that the pullbacks $\cO_{\tX}(i)$ of the line bundles~$\cO_{\P^5}(i)$ are contained in the subcategory~$\cD$ for all~$i$.
Moreover, the triple~$(\cO_\tX,\cO_\tX(1),\cO_\tX(2))$ restricts to an exceptional triple in each fiber of~$\cD$ over~$B$,
therefore by~\cite[Theorem~2]{Sam} it is a $B$-exceptional triple and there is a $B$-linear semiorthogonal decomposition
\begin{equation*}
\cD = \langle \cA, \tf^*\Db(B) \otimes \cO_\tX, \tf^*\Db(B) \otimes \cO_\tX(1), \tf^*\Db(B) \otimes \cO_\tX(2) \rangle
\end{equation*}
defining the $B$-linear admissible subcategory~$\cA \subset \cD \subset \Db(\tX)$,
where $\tf$ is the composition $\tX \to X \to B$.
Taking the base change of the above decomposition to fibers over all points~$b \ne o$ and comparing it with~\eqref{eq:sod-dby}
we deduce that~$\cA_{b}$ is the $K3$ category of~$X_b$.
Similarly, the definition of the fiber~$\cA_o$ of~$\cA$ at~$o$ coincides with the definition~\eqref{eq:cd-y} of the category~$\widetilde\cA_Y$.
Finally, this category is equivalent to~$\Db(S)$ by~\eqref{eq:tca-y}.
\end{proof}


\begin{thebibliography}{CKKM19}

\bibitem[Art74]{Artin}
M.~Artin.
\newblock Algebraic construction of {B}rieskorn's resolutions.
\newblock {\em J. Algebra}, 29:330--348, 1974.

\bibitem[Ati58]{At}
M.~F. Atiyah.
\newblock On analytic surfaces with double points.
\newblock {\em Proc. Roy. Soc. London Ser. A}, 247:237--244, 1958.

\bibitem[BK89]{BK}
A.~I. Bondal and M.~M. Kapranov.
\newblock Representable functors, {S}erre functors, and reconstructions.
\newblock {\em Izv. Akad. Nauk SSSR Ser. Mat.}, 53(6):1183--1205, 1337, 1989.

\bibitem[Bri71]{Bri}
E.~Brieskorn.
\newblock Singular elements of semi-simple algebraic groups.
\newblock In {\em Actes du {C}ongr\`es {I}nternational des {M}ath\'{e}maticiens
  ({N}ice, 1970), {T}ome 2}, pages 279--284. 1971.

\bibitem[BvdB03]{BV}
A.~Bondal and M.~van~den Bergh.
\newblock Generators and representability of functors in commutative and
  noncommutative geometry.
\newblock {\em Mosc. Math. J.}, 3(1):1--36, 258, 2003.

\bibitem[CKKM19]{CKKM}
Chiara Camere, Grzegorz Kapustka, Micha\l{} Kapustka, and Giovanni Mongardi.
\newblock Verra four-folds, twisted sheaves, and the last involution.
\newblock {\em Int. Math. Res. Not. IMRN}, (21):6661--6710, 2019.

\bibitem[GR17]{GR17}
Dennis Gaitsgory and Nick Rozenblyum.
\newblock {\em A study in derived algebraic geometry. {V}ol. {II}.
  {D}eformations, {L}ie theory and formal geometry}, volume 221 of {\em
  Mathematical Surveys and Monographs}.
\newblock American Mathematical Society, Providence, RI, 2017.

\bibitem[HLS16]{HLS}
Daniel Halpern-Leistner and Ian Shipman.
\newblock Autoequivalences of derived categories via geometric invariant
  theory.
\newblock {\em Adv. Math.}, 303:1264--1299, 2016.

\bibitem[KL15]{KL}
Alexander Kuznetsov and Valery~A. Lunts.
\newblock Categorical resolutions of irrational singularities.
\newblock {\em Int. Math. Res. Not. IMRN}, (13):4536--4625, 2015.

\bibitem[KP21a]{KP21}
Alexander Kuznetsov and Alexander Perry.
\newblock Categorical joins.
\newblock {\em J. Amer. Math. Soc.}, 34(2):505--564, 2021.

\bibitem[KP21b]{KP19}
Alexander Kuznetsov and Alexander Perry.
\newblock Homological projective duality for quadrics.
\newblock {\em J. Algebraic Geom.}, 30(3):457--476, 2021.

\bibitem[Kuz04]{K04:cubic}
Alexander Kuznetsov.
\newblock Derived category of a cubic threefold and the variety {$V_{14}$}.
\newblock {\em Tr. Mat. Inst. Steklova}, 246(Algebr. Geom. Metody, Svyazi i
  Prilozh.):183--207, 2004.

\bibitem[Kuz06]{K06}
Alexander Kuznetsov.
\newblock Hyperplane sections and derived categories.
\newblock {\em Izv. Ross. Akad. Nauk Ser. Mat.}, 70(3):23--128, 2006.

\bibitem[Kuz07]{K07}
Alexander Kuznetsov.
\newblock Homological projective duality.
\newblock {\em Publ. Math. Inst. Hautes \'{E}tudes Sci.}, (105):157--220, 2007.

\bibitem[Kuz08a]{K08b}
Alexander Kuznetsov.
\newblock Derived categories of quadric fibrations and intersections of
  quadrics.
\newblock {\em Adv. Math.}, 218(5):1340--1369, 2008.

\bibitem[Kuz08b]{K08c}
Alexander Kuznetsov.
\newblock Exceptional collections for {G}rassmannians of isotropic lines.
\newblock {\em Proc. Lond. Math. Soc. (3)}, 97(1):155--182, 2008.

\bibitem[Kuz08c]{K08}
Alexander Kuznetsov.
\newblock Lefschetz decompositions and categorical resolutions of
  singularities.
\newblock {\em Selecta Math. (N.S.)}, 13(4):661--696, 2008.

\bibitem[Kuz10]{K10:cubic}
Alexander Kuznetsov.
\newblock Derived categories of cubic fourfolds.
\newblock In {\em Cohomological and geometric approaches to rationality
  problems}, volume 282 of {\em Progr. Math.}, pages 219--243. Birkh\"{a}user
  Boston, Boston, MA, 2010.

\bibitem[Kuz11]{K11}
Alexander Kuznetsov.
\newblock Base change for semiorthogonal decompositions.
\newblock {\em Compos. Math.}, 147(3):852--876, 2011.

\bibitem[Kuz19]{K19}
Alexander Kuznetsov.
\newblock Calabi--{Y}au and fractional {C}alabi--{Y}au categories.
\newblock {\em J. Reine Angew. Math.}, 753:239--267, 2019.

\bibitem[Lun10]{Lunts10}
Valery~A. Lunts.
\newblock Categorical resolution of singularities.
\newblock {\em J. Algebra}, 323(10):2977--3003, 2010.

\bibitem[Per19]{P19}
Alexander Perry.
\newblock Noncommutative homological projective duality.
\newblock {\em Adv. Math.}, 350:877--972, 2019.

\bibitem[Sam07]{Sam}
Alexander Samokhin.
\newblock Some remarks on the derived categories of coherent sheaves on
  homogeneous spaces.
\newblock {\em J. Lond. Math. Soc. (2)}, 76(1):122--134, 2007.

\bibitem[SGA71]{SGA6}
{\em Th\'{e}orie des intersections et th\'{e}or\`eme de {R}iemann-{R}och}.
\newblock Lecture Notes in Mathematics, Vol. 225. Springer-Verlag, Berlin-New
  York, 1971.
\newblock S\'{e}minaire de G\'{e}om\'{e}trie Alg\'{e}brique du Bois-Marie
  1966--1967 (SGA 6), Dirig\'{e} par P. Berthelot, A. Grothendieck et L.
  Illusie. Avec la collaboration de D. Ferrand, J. P. Jouanolou, O. Jussila, S.
  Kleiman, M. Raynaud et J. P. Serre.

\bibitem[Tju70]{Tyurina}
G.~N. Tjurina.
\newblock Resolution of singularities of flat deformations of double rational
  points.
\newblock {\em Funkcional. Anal. i Prilo\v{z}en.}, 4(1):77--83, 1970.

\end{thebibliography}
\end{document}